\documentclass[11pt]{article}
 \usepackage{amsthm,mathrsfs,amsmath,amssymb,enumerate,color}
 \usepackage{hyperref}
 \usepackage{dsfont}
 \usepackage{comment}
 \usepackage{cite}
 \usepackage[T1]{fontenc}
 \usepackage[utf8]{inputenc}
 \usepackage{authblk}
 \usepackage{indentfirst}
 \usepackage{pgfplots}
 \pgfplotsset{compat=1.18}
 \numberwithin{equation}{section}
\newtheorem{theorem}{Theorem}[section]
\newtheorem{Proposition}[theorem]{Proposition}
\newtheorem{lemma}[theorem]{Lemma}

\theoremstyle{definition}

\newcommand{\RT}{{\mathbb{R}^3}}

\date{}
\title{\bf Multiple positive solutions of a quasilinear Schr\"{o}dinger-Poisson system with concave and convex nonlinearities}

 \author{{Lanxin Huang\footnote{Corresponding author. E-mail address: 812419761@qq.com.} \ \ \ \ \ Xinqi Zhou}\\
 {\small Department of Mathematics,  Kunming University of Science and Technology}\\
 {\small  Kunming 650500, China}}

\begin{document}
\maketitle

\setcounter{page}{1}\baselineskip 15pt

\begin{quote}{\small
{\bf Abstract } \  In this paper, we consider the   quasilinear Schr\"{o}dinger-Poisson system with concave and convex nonlinearities
  \begin{align*}
  \begin{cases}
  -\Delta_{p} u+\lambda V(x)|u|^{p-2}u + \mu \phi |u|^{p-2}u= a(x)|u|^{m-2}u + b(x)|u|^{q-2}u  & \ \ \ \mathrm{in}\ \mathbb{R}^{3},\\
  -\Delta \phi=|u|^{p}  &\ \ \ \mathrm{in}\ \mathbb{R}^{3},
  \end{cases}
 \end{align*}
 where $\lambda >0, ~\mu >0$, $\frac{3}{2}<p<3$, $1< q<p < m < 2p$
 and  $\Delta_{p} u= \hbox{div}(|\nabla u|^{p-2}\nabla u)$. We assume that $V(x) \in C(\mathbb{R}^{3}, \mathbb{R})$ is a steep potential well, while $a(x)$ and $b(x)$ are allowed to be sign-changing and satisfy some suitable assumptions in $\mathbb{R}^3$.
By using the Ekeland's variational principle and combining the constraint approach, we prove that the system admits two positive solutions. 
}
\end{quote}

\section{Introduction and main results}

In the past two decades, the classical semilinear   Schr\"{o}dinger-Poisson system
  \begin{align}\label{101}
  \begin{cases}
  -\Delta u+ V(x)u+\mu \phi u=f(x,u) &\ \ \ \mathrm{in}\ \mathbb{R}^{3},\\
  -\Delta \phi= u^2 &\ \ \ \mathrm{in}\ \mathbb{R}^{3},
  \end{cases}
  \end{align}
has been studied widely under variant assumptions on  the potential $V(x)$ and the nonlinearity $f(x,u)$.
  As originally proposed  by Benci and Fortunato in  \cite{1998BF,2002BF},  it models the interaction of a charged particle with an electrostatic field in quantum mechanics.
  For $V(x)=1$ and $f(x,u)=|u|^{m-2}u$, the general existence, multiplicity and non-existence results were established for $m\in(2,6)$, see \cite{2004DM-PRSEA,2004DM-ANS,2002d-ANS,2006Ruiz,2009YSD-NA}.
In \cite{2011JZ-JDE}, Jiang and Zhou considered  the  Schr\"{o}dinger-Poisson system
\begin{equation}\label{eqjz}
\begin{cases}
 -\Delta u+ (1+\lambda V(x))u+l(x) \phi u=|u|^{m-2}u &\ \ \ \mathrm{in}\ \mathbb{R}^{3},\\
  -\Delta \phi= u^2 &\ \ \ \mathrm{in}\ \mathbb{R}^{3},
\end{cases}
\end{equation}
where $\lambda>0$, $l(x)=\mu>0$, the potential $V(x)$ satisfies the following conditions,
 \begin{itemize}
 \item[$\mathrm{(H_{l})}$]  $V \in C(\mathbb{R}^3,\mathbb{R})$ with $V(x) \geqslant 0$ in $\mathbb{R}^3$ and there exists $c_0 > 0$ such that the set $\{V < c_0\} := \{x \in \mathbb{R}^3 \mid V(x) < c_0\}$ has finite positive Lebesgue measure;
 \item[$\mathrm{(H_{2})}$]   $\Omega = \text{int}\{x \in \mathbb{R}^3 \mid V(x) = 0\}$ is nonempty and has smooth boundary with $\overline{\Omega} = V^{-1}(0)$.
 \end{itemize}
They obtained the existence and concentration results for $m\in(2,3)\cup[4,6)$ by combining domains approximation with priori estimates.
This kind of hypotheses was first introduced by Bartsch and Wang \cite{1995BW-CPDE} in the study of a nonlinear
Schr\"{o}dinger equation and the potential $\lambda V(x)$ with $V$ satisfying $(H_{1})$--$(H_{2})$ is referred as the steep well potential.
Under the same conditions on $V(x)$, Zhao, Liu and Zhao \cite{2013ZLZ-JDE} studied  system (\ref{eqjz})  with  $l(x)\geqslant 0$ for $x\in \mathbb{R}^3$, and  obtained the existence and concentration results for $m\in(3,6)$ via variational methods. In particular, the potential $V(x)$ is allowed to be sign-changing for the case $m\in(4,6)$.
For related works of system (\ref{eqjz})   or similar problem including  steep well potential, we refer readers to \cite{2023PJ-ZAMP,2023KLT-JGA,2020SW-NA,2020YTW-AMC,2021SW-CCM}.

We now turn to the case of   system  (\ref{eqjz}) with   concave and convex nonlinearities. In \cite{2015SSZ-DCDS}, Sun, Su  and Zhao  considered the system
 \begin{eqnarray}
 \left\{
 \begin{array}{ll}
 \displaystyle  -\Delta u+\lambda V(x)u+\mu \phi u=k(x)|u|^{m-2}u+ h(x)|u|^{q-2}u\ \ \  &\mathrm{in} \ \mathbb{R}^{3},\\
 \displaystyle -\Delta \phi=u^2\ \ \  &\mathrm{in} \ \mathbb{R}^{3},
 \end{array}\label{103}
 \right. \end{eqnarray}
where $1< q < 2<m<3$, $\lambda =\mu=1$, and the nonnegative functions $V,~k,~h$ are  radial.
Following Ruiz's ideas in \cite{2006Ruiz}, they used the decay properties of functions in the radial symmetric space  to prove that the functional satisfies the Palais-Smale condition and is bounded from below.
Applying the variant version of Clark's theorem, it was proved in \cite{2015SSZ-DCDS}  that system {\rm(\ref{103})} admits infinitely many solutions with negative energy.
Subsequently, Sun and Wu in \cite{2021SW-CCM} investigated system {\rm(\ref{103})} for the wider range $2<m<4$ with a steep potential well $V(x)$ satisfying $(H_{1})$--$(H_{2})$.
Under suitable conditions on the parameters and functions, they obtained two positive solutions by using the Ekeland's variational principle
 and introducing the filtration of the Nehari manifold on the outside of a small ball.
Moreover,  their work removed the radial symmetry condition required in \cite{2015SSZ-DCDS}, and further addressed the more general case where the function $k(x)$ and $h(x)$ are allowed to be sign-changing. Furthermore, related works have also focused on sign-changing solutions. Specifically, Yang and Ou  in \cite{2021YO-JMAA}  obtained a sign-changing solution with positive energy of system  {\rm(\ref{103})} in a bounded domain for the case $\lambda=0$, $1< q < 2$ and  $4<m<6$, while Yang and Tang in \cite{2023YT-CAM} extended this result to the whole space $\mathbb{R}^{3}$ with $V(x)$ satisfies the condition
 \begin{itemize}
 \item[$\mathrm{(V)}$]  $V \in C(\mathbb{R}^3,\mathbb{R})$,  $\inf_{\mathbb{R}^3} V(x) \geqslant a >0$ in $\mathbb{R}^3$, and for any $c_0 > 0$, the set $\{x \in \mathbb{R}^3 \mid V(x) \leqslant c_0\}$ has finite positive Lebesgue measure.
 \end{itemize}

In a more general setting,  Shao and Mao in \cite{2018SM-AML,2019SM-JMP} studied the Schr\"{o}dinger-Poisson system with combined nonlinearities,
 \begin{eqnarray}
 \left\{
 \begin{array}{ll}
 \displaystyle  -\Delta u+ V(x)u+\phi u=\sigma g(x,u)+ \tau f(x,u)\ \ \  &\mathrm{in} \ \mathbb{R}^{3},\\
 \displaystyle -\Delta \phi=u^2\ \ \  &\mathrm{in} \ \mathbb{R}^{3},
 \end{array}\label{104}
 \right. \end{eqnarray}
where $\sigma$ and $\tau$ are parameters, $V(x)$ satisfies the above assumption $(V)$, the function $g$ is a concave term while  $f$ is a
convex term satisfying the following well-known (AR) condition
 \begin{itemize}
\item[$\mathrm{(A)}$]  there exists constants $\theta>4$,  such that for all   $x\in\mathbb{R}^3$ and  $u\in\mathbb{R}\setminus\{0\}$,
 $$0<\theta F(x,u)=\theta \int_{0}^{u}f(x,t)dt \leqslant uf(x,u).$$
 \end{itemize}
Applying the fountain theorem, they established the existence of infinitely many
solutions with high energy or small negative energy for system {\rm(\ref{104})}.
For related works of the system {\rm(\ref{104})} with concave and convex
nonlinearities, 
 we refer to \cite{2021-CL-AML,2017LS-CMA} and the references therein.


Recently, Du, Su and Wang \cite{2022DSW-JMAA,2022DSW-CPAA} first investigated the  quasilinear  Schr\"{o}dinger-Poisson  system
 \begin{eqnarray}
 \left\{
 \begin{array}{ll}
 \displaystyle -\Delta_{p} u+|u|^{p-2}u+\mu\phi |u|^{p-2}u=g(x,u)\ \ \  &\mathrm{in} \ \mathbb{R}^{3},\\
 \displaystyle -\Delta \phi=|u|^{p}\ \ \  &\mathrm{in} \ \mathbb{R}^{3},
 \end{array}\label{106}
 \right. \end{eqnarray}
where   $\mu >0$, $1<p<3$ and  $\Delta_{p} u= \hbox{div}(|\nabla u|^{p-2}\nabla u)$. In  \cite{2022DSW-JMAA}, they considered the subcritical case for $g(x,u)=|u|^{m-2}u$ with $p<m<p^{*}=\frac{3p}{3-p}$, and established the existence of nontrivial solutions of system {\rm(\ref{106})} for sufficiently small $\mu >0$.
In fact, they verified that the associated energy functional satisfies the mountain pass geometric structure by subtle scaling transformation, and obtained a Palais-Smale sequence at the mountain pass level. To further derive the boundedness of such Palais-Smale sequences, they adopted two distinct approaches according to different ranges of exponent  $m$. Precisely, for $\frac{2p(p+1)}{p+2} < m < 2p$, they employed the scaling technique by introducing an auxiliary functional; while for   $p < m \leqslant \frac{2p(p+1)}{p+2}$, they resorted to the cut-off method and constructed a truncated functional. When $\frac{3}{2}<p<3$ ($2p<p^{*}$),   Du, Su and Wang  \cite{2022DSW-CPAA} also considered the critical case for $g(x,u)=|u|^{p^{*}-2}u+\tau |u|^{m-2}u$ with $p<m<p^{*}$ and  $\tau >0$.  For sufficiently small $\mu >0$ and sufficiently large $\tau >0$, they established the existence of nontrivial solutions of system {\rm(\ref{106})}. Subsequently, further work on system {\rm(\ref{106})} concerning the normalized solutions, sign-changing solutions, multiple solutions, infinitely many solutions with varied nonlinear hypotheses, see for example \cite{2025XNL-ZAMP,2025LHR-JDE,2023CT-AML,2024HS-JMP}  and the references therein.

Nevertheless, few results have been established for  quasilinear Schr\"{o}dinger-Poisson systems with concave and convex nonlinearities including  the steep well potential. Motivated by the above observations, we investigate the following system
 \begin{align}\label{107}
  \begin{cases}
   -\Delta_{p} u+\lambda V(x)|u|^{p-2}u + \mu \phi |u|^{p-2}u= a(x)|u|^{m-2}u + b(x)|u|^{q-2}u  & \ \ \ \mathrm{in}\ \mathbb{R}^{3},\\
  -\Delta \phi=|u|^{p}  &\ \ \ \mathrm{in}\ \mathbb{R}^{3},
  \end{cases}
 \end{align}
 where $\lambda>0, ~\mu >0$, $\frac{3}{2}<p<3$,   $1< q<p < m < 2p$. The functions $a(x)$ and $b(x)$ satisfy the following conditions:
 \begin{itemize}
 \item[$\mathrm{(H_{3})}$]  $a \in L^\infty(\mathbb{R}^3)$ and $a(x) > 0$ in $\Omega$;
  \item[$\mathrm{(H_{4})}$]  $b \in L^{q_*}(\mathbb{R}^3)$ and $b^+ := \max \{b(x),0\} \not\equiv 0$, where $q_* = \frac{m}{m-q}$.
 \end{itemize}

Before stating our main result, we need to introduce some notations. Let $L^s(\mathbb{R}^3)$ denotes the Lebesgue space with the usual norm
  $|u|_{s}=\left(\int_{\mathbb{R}^3}|u|^{s}dx\right)^{\frac{1}{s}}$ for $1 \leqslant s<\infty$.    Let $D^{1,p}(\mathbb{R}^3)$ be the completion of $C_0^\infty(\mathbb{R}^3)$ with respect to the norm $\|u\|_{D^{1,p}}=\left(\int_{\mathbb{R}^3}|\nabla u|^pdx\right)^{\frac{1}{p}}.$   Denote by $\bar{S}_p$  the best constants for the embedding of $D^{1,p}(\mathbb{R}^3)$ in $L^{p^{*}}(\mathbb{R}^3)$.   
Let $W^{1,p}(\mathbb{R}^3)$ denotes the Sobolev space endowed with the  norm
 $\|u\|=\left(\int_{\mathbb{R}^3}|\nabla u|^p+|u|^pdx\right)^{\frac{1}{p}}.$
Let
$$
W_\lambda= \left\{ u \in W^{1,p}(\mathbb{R}^3) \mid \int_{\mathbb{R}^3} \lambda V(x)|u|^p \, dx < \infty \right\}
$$
be equipped with the  norm
\begin{equation}\label{2.1.1}
 \|u\|_\lambda = \left(\int_{\mathbb{R}^3}(|\nabla u|^p+\lambda  V(x)|u|^p)dx\right)^{\frac{1}{p}}.
\end{equation}
By virtue of the definition in \eqref{2.1.1},  it follows immediately that $\|u\| \leqslant \|u\|_\lambda$ for all $\lambda \geqslant 1$. 
 Denote by $S_{m,\Omega}$ be the best Sobolev constant for the embedding of $W_0^{1,p}(\Omega)$ in $L^m(\Omega)$ with $p < m< 2p$.  Denote by $\bar{S}_{p,\Omega}$  the best constants for the embedding of $D^{1,p}(\Omega)$ in $L^{p^{*}}(\Omega)$. 
 Denote  $|\cdot|$  the Lebesgue measure.
Set
\begin{equation}\label{1.1.1}
\mathbf{S}_m = \bar{S}_{p} |\{V < c_0\}|^{-\frac{p^{*}-m}{mp^{*}}}.
\end{equation}

 Now, we state our main results of this paper.

\begin{theorem}\label{thm1}
Let $\frac{3}{2}<p<3$ and $1< q<p < m < 2p$.  Assume that $(H_{1})$--$(H_{4})$ hold. Then there exists $\Pi_0 > 0$ such that for any $0 < \mu + |b|_{q_*} < \Pi_0$, system  \eqref{107} possesses  at least two positive solutions
 $(u_{\lambda,\mu}^\pm, \phi_{u_{\lambda,\mu}^\pm}) \in W_\lambda \times D^{1,2}(\mathbb{R}^3)$  satisfying
\[
\|u_{\lambda,\mu}^+\|_\lambda < \left( \frac{(p - q)\mathbf{S}_m^m}{(m - q)\|a\|_\infty} \right)^{\dfrac{1}{m-p}} < \|u_{\lambda,\mu}^-\|_\lambda 
\]
and
\[
\mathcal{J}_{\lambda,\mu}(u_{\lambda,\mu}^+) < 0 < \mathcal{J}_{\lambda,\mu}(u_{\lambda,\mu}^-)
\]
when  $\lambda > 0$ is sufficiently large, where $\mathcal{J}_{\lambda,\mu}$ is the corresponding energy functional of system \eqref{107} given by Section \ref{sec2}.
\end{theorem}

We now  outline our
strategy for establishing the existence of multiple positive  solutions to  system \eqref{107} and highlight the main difficulties addressed in this work.
On the one hand, influenced by the concave term $b(x)|u|^{q-2}u$, with the aid of Ekeland's variational principle, we  can seek a critical point with negative energy in a closed ball 
 $\overline{\mathrm{B}}_{\rho_0}:= \{u \in W_\lambda \mid \|u\|_\lambda \leqslant \rho_0\}$.
On the other hand, to derive a critical point with positive energy,  we adapt the methods of \cite{2021SW-CCM} to
introduce a filtration of the Nehari manifold outside this ball  as follows
\begin{equation}
N_{\lambda,\mu} [c] = \left\{ u \in N_{\lambda,\mu} \mid \|u\|_\lambda > \rho_0 \text{ and } J_{\lambda,\mu}(u) <c \right\},  
\end{equation}
 for some $c > 0$, where $N_{\lambda,\mu}$ is the Nehari manifold.
 In fact,  if we can further show  that the functional is bounded below  on the bounded part of the filtered set $N_{\lambda,\mu} [c]$, then we proceed to minimize the functional over  this bounded part.  Through this procedure, two distinct critical points of the functional can be obtained.

 
A principal difficulty arises in constructing a  proper  filtered set $N_{\lambda,\mu} [c]$. 
To overcome this, we need to impose suitable assumptions on  $\mu$ and $b(x)$ to seek
appropriate values for the radius $\rho_0$  and the upper bound $c$. 
Another key challenge lies in proving the non-emptiness of the bounded subset of the filtered set  $N_{\lambda,\mu} [c]$, which is  essential for  applying the minimization argument.



 The paper is organized as follows.   The variational framework of system \eqref{107} and some preliminary knowledge are given in Section \ref{sec2}.
 In Section  \ref{sec3}, we establish the existence of a local minimum for $J_{\lambda,\mu}$. In Section  \ref{sec4}, we propose the filtration of  Nehari manifold. In Section  \ref{sec5}, we prove Theorem  \ref{thm1}.

\section{Preliminaries}\label{sec2}
In this section, we give some useful preliminaries.
Adapting the approach developed in  \cite{2020SW-NA}, one has
$$
\int_{\mathbb{R}^3}(|\nabla u|^p+|u|^p) \, dx \leqslant \left(1 + \bar{S}_{p}^{-p}|\{V < c_0\}|^{\frac{p^{*}-p}{p^{*}}}\right) \|u\|_\lambda^p
$$
for all $\lambda$ satisfying
$$\lambda \geqslant c_0^{-1}\left(1 + \bar{S}_{p}^{-p}|\{V < c_0\}|^{\frac{p^{*}-p}{p^{*}}}\right)^{-1}.$$
This implies that the embedding $W_\lambda \hookrightarrow W^{1,p}(\mathbb{R}^3)$ is continuous and for $p <m < 2p$, we obtain
\begin{equation}\label{2.1.2}
\int_{\mathbb{R}^3} |u|^m \, dx  \leqslant   \mathbf{S}_{m} ^{-m} \|u\|_\lambda^m,  \quad \quad \forall  ~\lambda \geqslant \lambda_0 :=  c_0^{-1} \bar{S}_{p}^p |\{V < c_0\}|^{-\frac{p^{*}-p}{p^{*}}}. \end{equation}

For every $u \in W_\lambda$, the linear functional $\mathcal{T}_u : D^{1,2}(\mathbb{R}^3) \to \mathbb{R}$ is defined as
$$
\mathcal{T}_u(v) = \int_{\mathbb{R}^3} |u|^p v \, dx.
$$
By the H\"{o}lder inequality, one concludes
\begin{equation*}
|\mathcal{T}_u(v)| \leqslant \left( \int_{\mathbb{R}^3} |u|^{\frac{6p}{5}} dx \right)^{\frac{5}{6}} \left( \int_{\mathbb{R}^3} |v|^6 dx \right)^{\frac{1}{6}} \leqslant \bar{S}_2^{-1} \mathbf{S}_{\frac{6p}{5}} ^{-p} \|u\|_{\lambda}^p \|v\|_{D^{1,2}}.
\end{equation*}
Then, it follows that $\mathcal{T}_u$ is continuous on $D^{1,2}(\mathbb{R}^3)$. By the Lax-Milgram theorem, we know that there exists a unique $\phi_u \in D^{1,2}(\mathbb{R}^3)$ such that
$$
-\Delta \phi_u = |u|^p \quad \text{in } \mathbb{R}^3.
$$
According to \cite[Theorem 6.21]{LL}, $\phi_u$ has the following explicit expression
\begin{equation*}
\phi_u(x) = \frac{1}{4\pi} \int_{\mathbb{R}^3} \frac{|u(y)|^p}{|x-y|} dy \geqslant 0.
\end{equation*}
Moreover, $\phi_u$ has the following properties.

\begin{Proposition}[Proposition 2.1, \cite{2022DSW-JMAA}] \label{pro2.1}  For any  $u \in W_\lambda$, one has

$\mathrm{(i)}$ \ \ for all $t>0$, $\phi_{tu}=t^p \phi_u$;

$\mathrm{(ii)}$ \ \ $\|\phi_{u}\|_{D^{1,2}} \leqslant    \bar{S}_2^{-1} \mathbf{S}_{\frac{6p}{5}} ^{-p} \|u\|_{\lambda}^p $;

$\mathrm{(iii)}$ \ if $u_{n}\rightharpoonup u$ in $W_\lambda$, then $\phi_{u_{n}}\rightharpoonup \phi_{u}$ in $D^{1,2}(\mathbb{R}^3)$ and
\begin{equation*}
\int_{\mathbb{R}^3}\phi_{u_n} |u_n|^{p-2}u_n\varphi dx\rightarrow \int_{\mathbb{R}^3}\phi_u|u|^{p-2}u\varphi dx, \ \ \ \ \forall \ \varphi\in W_\lambda.
\end{equation*}
\end{Proposition}

Now, we establish the variational framework of \eqref{107}.
Arguing as in \cite{1998BF, 2002BF},  by Proposition \ref{pro2.1} and the implicit function theorem, the functional
\begin{eqnarray}\label{2.2}
\left. \begin{array}{ll}
\mathcal{J}_{\lambda,\mu}(u)&=\displaystyle\frac{1}{p}\int_{\mathbb{R}^3}(|\nabla u|^p+\lambda V(x)|u|^p)dx+\frac{\mu}{2p}\int_{\mathbb{R}^3}\phi_{u}|u|^pdx\\[3mm]
 &\ \ \ \displaystyle-\frac{1}{m}\int_{\mathbb{R}^3} a(x)|u|^{m}dx-\frac{1}{q}\int_{\mathbb{R}^3} b(x)|u|^qdx  \end{array}\right.
\end{eqnarray}
is a well-defined $C^1$ functional on $W_\lambda$ with  derivative
\begin{eqnarray*}\left. \begin{array}{ll}
\langle \mathcal{J}_{\lambda,\mu}'(u),v\rangle &=\displaystyle \int_{\mathbb{R}^3}(|\nabla u|^{p-2}\nabla u \nabla v+\lambda V(x)|u|^{p-2}u v)dx+\mu\int_{\mathbb{R}^3}\phi_u |u|^{p-2}u v dx\\[3mm]
 &\ \ \ \displaystyle-\int_{\mathbb{R}^3} a(x)|u|^{m-2}uvdx-\int_{\mathbb{R}^3} b(x)|u|^{q-2}uvdx,\ \ \ \ \forall \ u, v\in W_\lambda.
  \end{array}\right.
\end{eqnarray*}
Note that  the couple $(u,\phi_{u})\in W_\lambda\times D^{1,2}(\mathbb{R}^3)$ is a solution of \eqref{107} if and only if  $u \in W_\lambda$ is a critical point of $\mathcal{J}_{\lambda,\mu}$.
Therefore, finding a weak solution to the system \eqref{107} is equivalent to finding a critical point of the functional $\mathcal{J}_{\lambda,\mu}$ on $W_\lambda$.

\begin{lemma}\label{lem2.2}
 Let $\frac{3}{2}<p<3$   and   $1< q<p < m < 2p$. Assume that   $(\mathrm{H_{1}})$-$(\mathrm{H_{4}})$ hold.
For any bounded sequence $\{u_n\}\subset W_\lambda$ satisfying $J'(u_n)\rightarrow0$,
there exists $u\in W_\lambda$, such that, up to a subsequence, $\nabla u_n(x)\to\nabla u(x)$ a.e. in $\mathbb{R}^3$.
\end{lemma}

\begin{proof}[\bf Proof]  The proof is similar to that of  \cite[Lemma 3.1]{2022DSW-JMAA}, so we omit it here.
 \end{proof}


\begin{lemma}\label{lem2.3}
     Let $\frac{3}{2}<p<3$   and  $1< q<p < m < 2p$. Assume that $(H_{1})$, $(H_{3})$ and $(H_{4})$ hold. There exists a constant $c_1$  such that if $\lambda>0$ sufficiently large,  then any bounded sequence  $\{u_n\}\subset W_\lambda$ satisfying $\|u_n\|_\lambda \leqslant c_1$ and $\mathcal{J}_{\lambda,\mu}'(u_n)\rightarrow0$ has a strongly convergent subsequence.
\end{lemma}

\begin{proof}[\bf Proof]
 For sequence $\{u_n\} \subset W_\lambda$ satisfying  $\|u_n\|_\lambda  \leqslant c_1$,  there exist a subsequence $\{u_n\}$ and $u$ in $W_\lambda$ such that
 \begin{eqnarray}\label{2.3.1}
	\left \{\begin{array}{ll}
		\displaystyle u_{n}\rightharpoonup u  &~\mathrm{in} \ W_\lambda,\\
		\displaystyle u_{n}\rightarrow u  & ~\mathrm{in}\ L^{r}_{loc}(\RT), \   \ p \leqslant r<p^{*},\\
		u_{n}(x)\rightarrow u(x) &  ~\mathrm{a.e. \ in}~ \mathbb{R}^{3}.
	\end{array}
	\right.
\end{eqnarray}
Since   $\{|\nabla u_{n}|^{p-2}\nabla u_{n}\}$  is bounded in $L^{\frac{p}{p-1}}(\mathbb{R}^{3})$ and $|\nabla u_{n}(x)|^{p-2}\nabla u_{n}(x)\rightarrow |\nabla u(x)|^{p-2}\nabla u(x)$  a.e. in $\mathbb{R}^{3}$, using  \cite[Proposition 5.4.7]{2013Willem}, we deduce that
\begin{align*}
	|\nabla u_{n}|^{p-2}\nabla u_{n}\rightharpoonup |\nabla u|^{p-2}\nabla u   \ \ \ ~\mathrm{in}~ L^{\frac{p}{p-1}}(\mathbb{R}^{3}).
\end{align*}  
   Note that for any $v\in W_\lambda$
  , we have $\nabla v\in L^{p}(\mathbb{R}^{3})$, and then
  \begin{align}\label{u11}
  	\int_{\mathbb{R}^{3}} |\nabla u_{n}|^{p-2}\nabla u_{n}\nabla vdx\rightarrow  \int_{\mathbb{R}^{3}} |\nabla u|^{p-2}\nabla u\nabla vdx.
  \end{align}
  Proceeding as \eqref{u11}, we derive that
  \begin{eqnarray} \label{u2}
  	&& \int_{\mathbb{R}^{3}} \lambda V(x)|u_{n}|^{p-2} u_{n}vdx\rightarrow  \int_{\mathbb{R}^{3}} \lambda V(x)| u|^{p-2}u vdx,\\
  	&&   \int_{\mathbb{R}^{3}} a(x)|u_{n}|^{m-2}u_{n} vdx\rightarrow  \int_{\mathbb{R}^{3}}a(x) |u|^{m-2}u vdx, \label{u3}\\
  	&&   \int_{\mathbb{R}^{3}} b(x)|u_{n}|^{q-2}u_{n} vdx \rightarrow  \int_{\mathbb{R}^{3}} b(x)|u|^{q-2}u vdx.\label{u4}
  \end{eqnarray}
  By Proposition \ref{pro2.1}(iii), we have
  \begin{eqnarray} \label{u5}
  	\int_{\mathbb{R}^3} \phi_{u_n} |u_n|^{p-2}u_nv dx\rightarrow \int_{\mathbb{R}^3} \phi_u|u|^{p-2}uv dx.
  \end{eqnarray}
  Combining \eqref{u11}--\eqref{u5}, for $u_n\rightharpoonup u$ in $W_\lambda$, we obtain
  $$\langle \mathcal{J}_{\lambda,\mu}'(u_n),v\rangle \rightarrow \langle \mathcal{J}_{\lambda,\mu}'(u),v\rangle,$$
  which and $\mathcal{J}_{\lambda,\mu}'(u_n)\rightarrow 0$ imply that $\langle \mathcal{J}_{\lambda,\mu}'(u),v\rangle=0$ for any $v\in W_\lambda$.
  In particular,  \begin{align}\label{uu}
  	\langle \mathcal{J}_{\lambda,\mu}'(u),u\rangle=0
  \end{align}

 Now  we are going to prove that $u_{n}\rightarrow u$ strongly in $W_\lambda$.
Set $v_n=u_n-u$.  It follows from \eqref{2.3.1} that  $v_n\rightharpoonup 0$ in $W_\lambda$ and
 \begin{align}\label{2.3.11}
\|v_n\|_\lambda \leqslant  2c_1.
\end{align}
 Hence, by  the  Br\'{e}zis-Lieb Lemma (see \cite[Lemma 1.32]{1992BM}), we derive that
 \begin{eqnarray*}
 	&&   \|u_{n}\|_{\lambda}^{p}= \|u\|_{\lambda}^{p}+ \|v_{n}\|_{\lambda}^{p}+o(1),\\[2mm]
 	&&   \int_{\mathbb{R}^{3}}  a(x)|u_{n}|^{m}dx=\int_{\mathbb{R}^{3}}  a(x)|u|^{m}dx+\int_{\mathbb{R}^{3}} a(x) |v_{n}|^{m}dx+o(1), \\[2mm]
 	&&   \int_{\mathbb{R}^{3}} b(x)|u_{n}|^{q}dx=\int_{\mathbb{R}^{3}} b(x)|u|^{q}dx+\int_{\mathbb{R}^{3}} b(x)|v_{n}|^{q}dx+o(1).
 \end{eqnarray*}
 Moreover, by \cite[Lemma 2.2]{2025LH-BMS}, we have
  \begin{eqnarray*}
    \mu \int_{\mathbb{R}^{3}} \phi_{u_{n}}|u_{n}|^pdx=  \mu \int_{\mathbb{R}^{3}} \phi_{u}|u|^pdx+   \mu\int_{\mathbb{R}^{3}} \phi_{v_{n}}|v_{n}|^pdx+o(1).
 \end{eqnarray*}
 This implies that
  \begin{equation*}
 	\begin{aligned}
   \displaystyle	\langle \mathcal{J}_{\lambda,\mu}'(u_n),u_n\rangle &=\displaystyle \|u_{n}\|_{\lambda}^{p}+\mu \int_{\mathbb{R}^3}\phi_{u_{n}}|u_{n}|^pdx-\int_{\mathbb{R}^3} a(x)|u_{n}|^{m}dx-\int_{\mathbb{R}^3} b(x)|u_{n}|^qdx\\
    & = \displaystyle  \langle  \mathcal{J}_{\lambda,\mu}'(u),u\rangle+ \displaystyle  \langle  \mathcal{J}_{\lambda,\mu}'(v_n), v_n\rangle+o(1),
 \end{aligned}
 \end{equation*}
  which and \eqref{uu} imply that 
  \begin{align}\label{u123}
	\langle \mathcal{J}'(v_n), v_n \rangle=o(1).
\end{align}
Further, by  $(H_{1})$ and \eqref{2.3.1}, it is concluded that
	\[
	\int_{\mathbb{R}^3} |v_n|^p dx \leqslant  \frac{1}{\lambda c_0} \int_{\mathbb{R}^3} \lambda V(x) |v_n|^p dx + \int_{\{V<c_0\}} |v_n|^p dx  \leqslant \frac{1}{\lambda c_0} \|v_n\|_\lambda^p + o(1),
	\]
and
 \begin{equation} \label{u124}
	\begin{aligned}
	 \displaystyle	\int_{\mathbb{R}^3} |v_n|^m dx &\leqslant   \displaystyle \left( \int_{\mathbb{R}^3} |v_n|^p dx \right)^{\frac{m(p-3)+3p}{p^{2}}} 
		\left[ \bar{S}_{p}^{-p^{*}} \left( \int_{\mathbb{R}^3} |\nabla v_n|^{p} dx \right)^{\frac{p^{*}}{p}} \right]^{\frac{(m-p)(3-p)}{p^{2}}} + o(1)  \\
		&\leqslant  \displaystyle \left( \frac{1}{\lambda c_0} \right)^{\frac{m(p-3)+3p}{p^{2}}} \bar{S}_{p}^{-\frac{3(m-p)}{p}} \|v_n\|_\lambda^m + o(1).
 \end{aligned}
\end{equation}
  Proceeding as \eqref{u11}, we derive that
	\begin{equation}
		\int_{\mathbb{R}^3} b(x) |v_n|^q dx = o(1). \label{3.3}
	\end{equation}
	Thus, by condition $(H_{3})$ and \eqref{2.3.11}--\eqref{3.3} one has
 \begin{equation*} 
	\begin{aligned}
		o(1) &=  \displaystyle \langle \mathcal{J}_{\lambda,\mu}'(v_n), v_n \rangle \\
		&\geqslant  \displaystyle \|v_n\|_\lambda^p - \left( \frac{1}{\lambda c_0} \right)^{\frac{m(p-3)+3p}{p^{2}}} \frac{\|a\|_\infty}{\bar{S}_{p}^{\frac{3(m-p)}{p}}} \|v_n\|_\lambda^m + o(1) \\[2mm] 
		&\geqslant  \displaystyle \|v_n\|_\lambda^p - \left( \frac{1}{\lambda c_0} \right)^{\frac{m(p-3)+3p}{p^{2}}} \frac{2^m c_1^{m}  \|a\|_\infty}{\bar{S}_{p}^{\frac{3(m-p)}{p}}} + o(1).
\end{aligned}
\end{equation*}
Therefore, we conclude that $v_n \to 0$ strongly in $W_\lambda$ for $\lambda > 0$ sufficiently large. The proof is complete.
\end{proof}

\section{A solution with negative energy}\label{sec3}

	In this section, we  will search for a positive solution to \eqref{107} with negative energy for sufficiently large  $\lambda  > 0$  and sufficiently small  $|b|_{q_*}$. With the
	aid of Ekeland’s variational principle, such a solution is constructed as a local minimizer of the energy functional $\mathcal{J}_{\lambda,\mu}$.
	
\begin{lemma}\label{lem3.1}
 Let $\frac{3}{2}<p<3$   and  $1< q<p < m < 2p$. Assume that   $(H_{1})$, $(H_{3})$ and $(H_{4})$  hold. There exist constants  $\alpha > 0$, $\rho_0> 0$ and $\Gamma_0 > 0$  such that if  $\lambda \geqslant \lambda_0$ and $0 < |b|_{q_*} <  \Gamma_0$, then $ \mathcal{J}_{\lambda,\mu}(u) \geqslant \alpha$ with $ \|u\|_\lambda = \rho_0$. 
\end{lemma}

\begin{proof}[\bf Proof]
For $u \in W_\lambda$ and $\lambda \geqslant \lambda_0$, by \eqref{2.1.2}--\eqref{2.2} and the H\"{o}lder inequality, one has,
\begin{equation*}
	\begin{aligned}
 \mathcal{J}_{\lambda,\mu}(u)  &\geqslant \frac{1}{p}\|u\|_\lambda^p-\frac{\|a\|_\infty}{m \mathbf{S}_m^m}\|u\|_\lambda^m -\frac{|b|_{q_*}}{q \mathbf{S}_m^q}\|u\|_\lambda^q  \\
& =\displaystyle  \|u\|_\lambda ^q \left( \frac{1}{p}\|u\|_\lambda^{p-q} - \frac{\|a\|_\infty}{m \mathbf{S}_m^m}\|u\|_\lambda^{m-q} - \frac{|b|_{q_*}}{q \mathbf{S}_m^q}\right).
\end{aligned}
\end{equation*}
Setting  
\[
f(t) = \frac{1}{p} t^{p-q} - \frac{\|a\|_\infty}{m \mathbf{S}_m^m} t^{m-q},~  ~~\forall ~t>0.
\]
A direct calculation shows that $f$ attains its maximum at
\[
t_0 = \left[ \frac{m(p-q) \mathbf{S}_m^m}{p(m-q)\|a\|_\infty} \right]^\frac{1}{m-p} > 0,
\]
and $\max_{t \geqslant 0} f(t) =f(t_0)$.
For the chosen radius
\[
\rho_0 = \left( \frac{(p-q) \mathbf{S}_m^m}{(m-q)\|a\|_\infty} \right)^\frac{1}{m-p},
\]
one checks that $\rho_0 < t_0$ and
\[
f(\rho_0) = \frac{(m-p)(m+p-q)}{pm(m-q)} \left( \frac{(p-q) \mathbf{S}_m^m}{(m-q)\|a\|_\infty} \right)^\frac{p-q}{m-p} > 0.
\]
Then we conclude that, for  $\|u\|_\lambda=\rho_0$,
\begin{align*}
\mathcal{J}_{\lambda,\mu}(u)
\geqslant   \rho_0^q \left( f(\rho_0) - \frac{|b|_{q_*}}{q \mathbf{S}_m^q} \right).
\end{align*}
Choosing 
$$\Gamma_0 = \frac{q(m-p)(m+p-q)\mathbf{S}_m^q}{2pm(m-q)} \left( \frac{(p-q)\mathbf{S}_m^m}{(m-q)\|a\|_\infty} \right)^{\frac{p-q}{m-p}},
 ~~~\alpha =\frac{ \rho_0^q f(\rho_0)}{2}.
$$
Therefore, it follows that  $\mathcal{J}_{\lambda,\mu}(u) \geqslant \alpha$ for 
$|b|_{q_*} < \Gamma_0$ and  $\lambda \geqslant \lambda_0$.
\end{proof}

\begin{theorem}\label{thm3.2}
	Let $\frac{3}{2} < p < 3$ and $1 < q < p < m < 2p$. Assume that  $(H_{1})$-$(H_{4})$ hold.
	Then for sufficiently large  $\lambda > 0$  and $0 < |b|_{q_*} < \Gamma_0$, system \eqref{107} admits a positive
	solution  $u_{\lambda,\mu}^+ \in W_\lambda$ which has negative energy and satisfies $\|u_{\lambda,\mu}^+\|_\lambda < \rho_0$,
 where $ \Gamma_0$  and $\rho_0$ are given by Lemma \ref{lem3.1}.
\end{theorem}

\begin{proof}[\bf Proof]
 By $(H_{4})$, we can choose a function $\omega \in W_\lambda \setminus\{0\}$ such that
$$
\int_{\mathbb{R}^3} b(x)|\omega|^q dx > 0.
$$
When $t > 0$ is sufficiently small, one has
\begin{equation*}
\begin{aligned}
\mathcal{J}_{\lambda,\mu}(t\omega) = \frac{t^p}{p}\|\omega\|_\lambda^p + \frac{\mu t^{2p}}{2p}\int_{\mathbb{R}^3} \phi_\omega|\omega|^p dx - \frac{t^m}{m}\int_{\mathbb{R}^3} a(x)|\omega|^m dx - \frac{t^q}{q}\int_{\mathbb{R}^3} b(x)|\omega|^q dx
< 0.
\end{aligned}
\end{equation*}
This shows that 
\begin{equation}\label{3.2.1}
c^*=  \inf_{u\in  \bar{B}_{\rho_{0}}}\mathcal{J}_{\lambda,\mu}(u)<0,
\end{equation}
where $ \bar{B}_{\rho_{0}}=\{u \in  W_\lambda:  \|u\|_{\lambda} \leqslant \rho_{0} \}$
 and $\rho_0$ is given by Lemma \ref{lem3.1}.
Using the Ekeland's variational principle \cite{2011DPV}, there exists a sequence $\{u_n\} \subset \bar{B}_{\rho_{0}}$ such that
\[
\mathcal{J}_{\lambda,\mu}(u_n) =c^*+ o(1), \quad \quad \mathcal{J}_{\lambda,\mu}'(u_n) = o(1) \text{ in } W_\lambda^{-1}.
\]
Then Lemma \ref{lem2.3} implies that  for sufficiently large $\lambda>0$, there exists a local minimizer  $u_{\lambda,\mu}^+ \in \bar{B}_{\rho_{0}}$ of the functional $\mathcal{J}_{\lambda,\mu}$, and the sequence $\left\lbrace u_n \right\rbrace $ satisfies
  $u_n \to u_{\lambda,\mu}^+$ strongly in $W_\lambda$ with
  \[
  \mathcal{J}_{\lambda,\mu}(u_{\lambda,\mu}^+) =c^*< 0, \quad \quad \mathcal{J}_{\lambda,\mu}'(u_{\lambda,\mu}^+) =0.
  \]
 Since $\mathcal{J}_{\lambda,\mu}(u_{\lambda,\mu}^+) = \mathcal{J}_{\lambda,\mu}(|u_{\lambda,\mu}^+|) =c^*$, we may assume that $u_{\lambda,\mu}^+$ is a positive solution of system \ref{107}.   The proof is complete.
\end{proof}

\section{The filtration of the Nehari manifold}\label{sec4}

In this section, we define the Nehari manifold and construct its associated filtration. We first introduce the Nehari manifold as
\[
N_{\lambda,\mu} := \{ u \in  W_\lambda \backslash \{0\} \mid \langle \mathcal{J}_{\lambda,\mu}'(u), u \rangle = 0 \}.
\]
Then $u \in N_{\lambda,\mu}$ if and only if
\begin{align}\label{4.0.1}
\|u\|_\lambda^p + \mu \int_{\mathbb{R}^3} \phi_u |u|^p dx - \int_{\mathbb{R}^3} a(x)|u|^m dx - \int_{\mathbb{R}^3} b(x)|u|^q dx = 0.
\end{align}
Note that  $N_{\lambda,\mu}$ is closely linked to the fiber map $K_{\lambda,u}: t \mapsto \mathcal{J}_{\lambda,\mu}(tu)$, defined by 
\[
K_{\lambda,u}(t) = \frac{t^p}{p}\|u\|_\lambda^p + \frac{\mu t^{2p}}{2p} \int_{\mathbb{R}^3} \phi_u |u|^p dx - \frac{t^m}{m} \int_{\mathbb{R}^3} a(x)|u|^m dx - \frac{t^q}{q} \int_{\mathbb{R}^3} b(x)|u|^q dx, ~~~\forall~t>0.
\]
 Moreover, for $u \in W_\lambda$, we have
\[
K_{\lambda,u}'(t) = t^{p-1}\|u\|_\lambda^p + \mu t^{2p-1} \int_{\mathbb{R}^3} \phi_u |u|^p dx - t^{m-1} \int_{\mathbb{R}^3} a(x)|u|^m dx - t^{q-1} \int_{\mathbb{R}^3} b(x)|u|^q dx,
\]
and
\begin{align*}
	\begin{aligned}
	\displaystyle K_{\lambda,u}''(t) &=\displaystyle (p-1) t^{p-2}\|u\|_\lambda^p + (2p-1)\mu t^{2p-2} \int_{\mathbb{R}^3} \phi_u |u|^p dx\\
	&\ \ \ \ \displaystyle  - (m-1)t^{m-2} \int_{\mathbb{R}^3} a(x)|u|^m dx - (q-1)t^{q-2} \int_{\mathbb{R}^3} b(x)|u|^q dx.
\end{aligned}
\end{align*}
It is easy to see that $tu \in N_{\lambda,\mu}$  if and only if $K_{\lambda,u}'(t)=0$ holds. Particularly, $u \in N_{\lambda,\mu}$ if and only if  $K_{\lambda,u}'(1)=0$ holds. So, $N_{\lambda,\mu}$ can be splitted into three parts corresponding to the local minima, local maxima and points of inflection. According to \cite{1992Ta}, we define
\[
\begin{aligned}
	N_{\lambda,\mu}^+ &= \{ u \in N_{\lambda,\mu} \mid K_{\lambda,u}''(1) > 0 \}, \\
	N_{\lambda,\mu}^0 &= \{ u \in N_{\lambda,\mu} \mid K_{\lambda,u}''(1) = 0 \}, \\
	N_{\lambda,\mu}^- &= \{ u \in N_{\lambda,\mu} \mid K_{\lambda,u}''(1) < 0 \}.
\end{aligned}
\]
Then we have the following conclusion which is similar to Theorem 2.3 in \cite{2003Br}.
\begin{lemma}\label{lem4.1}
   Let $\frac{3}{2}<p<3$ and
    $1<q<p<m<2p$. 	Assume that  $u_0$ is a local minimizer for $\mathcal{J}_{\lambda,\mu}$ on $N_{\lambda,\mu}$ with $u_0 \notin N_{\lambda,\mu}^0$. Then $\mathcal{J}_{\lambda,\mu}'(u_0)=0$ in $W_{\lambda}^{-1}$.
\end{lemma}

We next turn to establishing the uniform lower bound of the functional $\mathcal{J}_{\lambda,\mu}$ restricted to   $ N_{\lambda,\mu}^-$, 
which is crucial for constraining  $\mathcal{J}_{\lambda,\mu}$ on the special filtered set to prove the existence of local minimizers later on. 

\begin{lemma}\label{lem4.2}
Let $\frac{3}{2}<p<3$ and $1<q<p<m<2p$. Assume that  $(H_{1})$, $(H_{3})$ and $(H_{4})$ hold. There exist constants $ \Gamma_1 >0$, $D := D(p,q,m,V,a,b) > 0$, such that if  $\lambda \geqslant \lambda_0$ and  $0 < |b|_{q_*} < \Gamma_1$, then 
    $\mathcal{J}_{\lambda,\mu}(u) > D$ for  $u \in N_{\lambda,\mu}^-$.
\end{lemma}

\begin{proof}
For each $u \in N_{\lambda,\mu}$, we have
\begin{align}
	K_{\lambda,u}''(1)
	& = (p-q)\|u\|_\lambda^p + (2p-q)\mu \int_{\mathbb{R}^3} \phi_u |u|^p  dx - (m-q)\int_{\mathbb{R}^3} a(x)|u|^m dx \label{4.1} \\
	& = -(m-p)\|u\|_\lambda^p + (2p-m)\mu \int_{\mathbb{R}^3} \phi_u |u|^p dx + (m-q)\int_{\mathbb{R}^3} b(x)|u|^q dx \label{4.2}\\
	& = -p\|u\|_\lambda^p + (2p-m)\int_{\mathbb{R}^3} a(x)|u|^m dx + (2p-q)\int_{\mathbb{R}^3} b(x)|u|^q dx.\label{4.3}
\end{align}
Let $u \in N_{\lambda,\mu}^-$, it follows from \eqref{2.1.2} and \eqref{4.1} that
\[
(p-q)\|u\|_\lambda^p < (p-q)\|u\|_\lambda^p + (2p-q)\mu \int_{\mathbb{R}^3} \phi_u |u|^p  dx  < (m-q)\mathbf{S}_m^{-m}\|a\|_\infty\|u\|_\lambda^m,
\]
which implies that
\begin{equation}
	\|u\|_\lambda > \rho_0= \left( \frac{(p-q)\mathbf{S}_m^m}{(m-q)\|a\|_\infty} \right)^\frac{1}{m-p}. \label{4.4}
\end{equation}
Moreover, for  $\lambda \geqslant \lambda_0$, it follows from  \eqref{4.0.1}, \eqref{4.3}--\eqref{4.4}  that
\begin{align*}
	\mathcal{J}_{\lambda,\mu}(u) &= \frac1p\|u\|_\lambda^p+\frac{\mu}{2p}\int_{\mathbb{R}^3}\phi_u|u|^pdx-\frac1m\int_{\mathbb{R}^3}a(x)|u|^mdx-\frac1q\int_{\mathbb{R}^3}b(x)|u|^qdx \\[2mm]
	&=\frac1{2p}\|u\|_\lambda^p+\frac{m-2p}{2pm}\int_{\mathbb{R}^3}a(x)|u|^mdx+\frac{q-2p}{2pq}\int_{\mathbb{R}^3}b(x)|u|^qdx \\[2mm]
	&> \frac{m-p}{2pm}\|u\|_\lambda^p - \frac{(2p-q)(m-q)}{2pqm}\int_{\mathbb{R}^3} b(x)|u|^q dx \\[2mm]
	&\geqslant \left[ \frac{m-p}{2pm}\left( \frac{(p-q)\mathbf{S}_m^m}{(m-q)\|a\|_\infty} \right)^{\frac{p-q}{m-p}} - \frac{(2p-q)(m-q)|b|_{q_*}}{2pqm \mathbf{S}_m^q} \right] \|u\|_\lambda^q.
\end{align*}
Choosing 
$$ \Gamma_1 :=\min \left\lbrace  \frac{q(m-p)}{2(2p-q)} \left( \frac{(p-q)^{\frac{p-q}{m-q}} \mathbf{S}_m^{p}}{(m-q)\|a\|_\infty^{\frac{p-q}{m-q}}} \right)^{\frac{m-q}{m-p}},\Gamma_0 \right\rbrace,
$$
and 
$$~~~D: =\frac{ \rho_0^q(m-p)}{4pm}\left( \frac{(p-q)\mathbf{S}_m^m}{(m-q)\|a\|_\infty} \right)^{\frac{p-q}{m-p}}.$$
Therefore, it follows that  $\mathcal{J}_{\lambda,\mu}(u) >D>0$ for  $u \in N_{\lambda,\mu}^-$ with 
$|b|_{q_*} < \Gamma_1$ and  $\lambda \geqslant \lambda_0$.
\end{proof}

Now, we are going to establish the filtration of the Nehari manifold. Before that, we need to introduce some notations.
Let
\[
A(p,m) = \left( \frac{p}{2p-m} \right)^{\frac{1}{m-p}},
\]
and
\[
B(p,q,m) = \frac{m(2p-q)2^{\frac{2q}{m-p}}}{q(m-q)S_{m,\Omega}^{q-m}} \left( \frac{(p-q)(2p-m)}{m-q} \right)^{\frac{p-q}{m-p}}.
\]
For notational  convenience, we further define
\[
M(p,q,m) = A^{p}(p,m)(1 + B(p,q,m))\int_{\Omega} a(x)\omega_{\Omega}^{m}dx,
\]
where $\omega_{\Omega}$ is the positive ground state solution of  equation \eqref{4.8}.  We define
\begin{equation}\label{4.1.5}
	\begin{aligned}
		N_{\lambda,\mu}\left[ \frac{m-p}{pm} M(p,q,m)  \right]
		= \left\{ u \in N_{\lambda,\mu} \mid \|u\|_\lambda > \rho_0 \text{ and } \mathcal{J}_{\lambda,\mu}(u) < \frac{m-p}{pm} M(p,q,m)) \right\}.
	\end{aligned}
\end{equation}
For any $u \in N_{\lambda,\mu}\left[ \frac{m-p}{pm} M(p,q,m)  \right]$,
by  Proposition \ref{pro2.1} (ii), \eqref{2.1.2}  and \eqref{4.0.1},  we get
\begin{align*}
	\frac{m-p}{pm} M(p,q,m) 
	&> \mathcal{J}_{\lambda,\mu}(u) \\[3mm]
	&= \frac{m-p}{pm}\|u\|_\lambda^p - \frac{2p-m}{2pm}\mu\int_{\mathbb{R}^3}\phi_u|u|^p dx - \frac{m-q}{mq}\int_{\mathbb{R}^3}b(x)|u|^q dx \\[3mm]
	&\geqslant \frac{m-p}{pm}\|u\|_\lambda^p - \frac{(2p-m)\bar{S}_2^{-2} \mathbf{S}_{\frac{6p}{5}} ^{-2p}}{2pm}\mu\|u\|_\lambda^{2p} - \frac{m-q}{mq\mathbf{S}_m^q}|b|_{q_*}\|u\|_\lambda^q \\[3mm]
	&> \frac{m-p}{pm}\|u\|_\lambda^p - \left[ \frac{(2p-m)\bar{S}_2^{-2} \mathbf{S}_{\frac{6p}{5}} ^{-2p}}{2pm}\mu + \frac{m-q}{mq\mathbf{S}_m^q}\rho_0^{q-2p}|b|_{q_*} \right]\|u\|_\lambda^{2p}.
\end{align*}
Hence, if we take
\begin{equation}\label{4.5.1}
0 <\frac{(2p-m)\bar{S}_2^{-2} \mathbf{S}_{\frac{6p}{5}} ^{-2p}}{2pm}\mu + \frac{m-q}{mq\mathbf{S}_m^q}\rho_0^{q-2p}|b|_{q_*}< \Gamma_2,
\end{equation}
where
\begin{equation}\label{4.5.2}
	\Gamma_2 := \min\left\{ \frac{m-p}{4pmM(p,q,m) }, \frac{m-p}{2pm}\left( \frac{(2p-m)\|a\|_\infty}{p\mathbf{S}_m^m} \right)^{\frac{p}{m-p}} \right\},
\end{equation}
then there exist two positive numbers $\widehat{D}_1,~ \widehat{D}_2$ satisfying
$$\widehat{D}_1 <\max \left\{ \left( \frac{p\mathbf{S}_m^m}{(2p-m)\|a\|_\infty} \right)^{\frac{1}{m-p}},\left(2M(p,q,m) \right)^{\frac{1}{p}} \right\}< \widehat{D}_2$$
such that either \( \|u\|_\lambda < \widehat{D}_1 \) or \( \|u\|_\lambda > \widehat{D}_2 \). Therefore,  there holds
\begin{equation}\label{4.5}
	\begin{aligned}
		N_{\lambda,\mu}\left[ \frac{m-p}{pm} M(p,q,m)  \right]
		=N_{\lambda,\mu}^{(1)} \cup N_{\lambda,\mu}^{(2)},
	\end{aligned}
\end{equation}
where
\[
N_{\lambda,\mu}^{(1)}:=\left\{u\in N_{\lambda,\mu}\left[\frac{m-p}{pm} M(p,q,m) \right]\mid\|u\|_\lambda<\widehat{D}_1\right\}
\]
and
\[
N_{\lambda,\mu}^{(2)}:=\left\{u\in N_{\lambda,\mu}\left[\frac{m-p}{pm} M(p,q,m) \right]\mid\|u\|_\lambda>\widehat{D}_2\right\}.
\]

Next, we will claim  $N_{\lambda,\mu}^{(1)}\subset N_{\lambda,\mu}^-$. 
In view of $q<2p$, for every $u \in N_{\lambda,\mu}^{(1)}$, by   \eqref{4.2}  and \eqref{4.4}, we get
\begin{equation}\label{4.7}
	\begin{aligned}
		K_{\lambda,u}''(1)
		&= -(m-p)\|u\|_\lambda^p + (2p-m)\mu
		\int_{\mathbb{R}^3} \phi_u |u|^p dx + (m-q)\int_{\mathbb{R}^3} b(x)|u|^q dx \\
		&\leqslant -(m-p)\|u\|_\lambda^p+(2p-m)\bar{S}_2^{-2} \mathbf{S}_{\frac{6p}{5}} ^{-2p}\mu\|u\|_\lambda^{2p}+\frac{m-q}{\mathbf{S}_m^q}|b|_{q_*}\|u\|_\lambda^q\\
		&<-(m-p)\|u\|_\lambda^p+\left[(2p-m)\bar{S}_2^{-2} \mathbf{S}_{\frac{6p}{5}} ^{-2p}\mu+\frac{m-q}{\mathbf{S}_m^q}\rho_0^{q-2p}|b|_{q_*}\right]\|u\|_\lambda^{2p}.
	\end{aligned}
\end{equation}
Furthermore, by \eqref{4.5.1}--\eqref{4.5.2}, it is easy to see that
\begin{align*}
	&\frac{(2p-m)\bar{S}_2^{-2} \mathbf{S}_{\frac{6p}{5}} ^{-2p}}{2p}\mu+\frac{m-q}{2p\mathbf{S}_m^q}\rho_0^{q-2p}|b|_{q_*}\\[3mm]
	&<\frac{(2p-m)\bar{S}_2^{-2} \mathbf{S}_{\frac{6p}{5}} ^{-2p}}{2p}\mu+\frac{m-q}{q\mathbf{S}_m^q}\rho_0^{q-2p}|b|_{q_*}\\[3mm]
	&<\min\left\{ \frac{m-p}{4pM(p,q,m) }, \frac{m-p}{2p}\left( \frac{(2p-m)\|a\|_\infty}{p\mathbf{S}_m^m} \right)^{\frac{p}{m-p}} \right\}.
\end{align*}
Combining this  with $\|u\|_\lambda<\widehat{D}_1$, one yields that
\begin{equation}\label{4.1.7}
		K_{\lambda,u}''(1)<0.
\end{equation}
Hence, we have the following conclusion.
\begin{lemma}\label{lem4.3}
	Let  $\frac{3}{2}<p<3$ and $1< q<p < m < 2p$. Assume that  $(H_{1})$-$(H_{4})$ hold.  There exists a constant  $\Pi_1>0$ such that if $\lambda \geqslant \lambda_0$ and $0<\mu+|b|_{q_*}<\Pi_1$,  then $N_{\lambda,\mu}^{(1)}\subset N_{\lambda,\mu}^-$
	is a \(C^1\) sub-manifold. Moreover, any local minimizer of the functional $\mathcal{J}_{\lambda,\mu}$ in $N_{\lambda,\mu}^{(1)}$ is a critical point of $\mathcal{J}_{\lambda,\mu}$ in $W_\lambda$.
\end{lemma}

In the following, we prove that \(N_{\lambda,\mu}^{(1)}\) is non-empty. Let \(w_\Omega\) be the positive ground state solution of
\begin{equation}\label{4.8}
	\begin{cases}
		-\Delta_{p} u= a(x)|u|^{m-2}u  & \ \ \ \mathrm{in}\ \Omega,\\
		u\in W_0^{1,p}(\Omega),
	\end{cases}
\end{equation}
where $p<m<2p$ and $\Omega$ is as in condition $(H_{2})$.  The functional \(\mathcal{J}_\Omega\) defined on \(W_0^{1,p}(\Omega)\) by
\begin{equation}
	\mathcal{J}_\Omega(u)=\frac{1}{p}\int_\Omega|\nabla u|^pdx-\frac{1}{m}\int_\Omega  a(x)|u|^mdx, \label{4.9}
\end{equation}
We also define its  associated Nehari manifold
\[
N_\Omega=\{u\in W_0^{1,p}(\Omega)\setminus\{0\}\mid\langle \mathcal{J}_\Omega'(u),u\rangle=0\}.
\]
Then, it follows from $(H_{2})$ that
\begin{equation}\label{4.9}
\|w_\Omega\|_\lambda^p= \int_\Omega |\nabla w_\Omega|^p dx=\int_\Omega a w_\Omega^mdx,
\end{equation}
and
\[
\inf_{u\in N_\Omega}\mathcal{J}_\Omega(u)=\mathcal{J}_\Omega(w_\Omega)=\frac{m-p}{pm}\int_\Omega a(x)w_\Omega^mdx.
\]

\begin{lemma}\label{lem4.4}
	Let  $\frac{3}{2}<p<3$ and  $1< q<p < m < 2p$. Assume that  $(H_{1})$-$(H_{4})$ hold. There exists a constant $\Pi_2$ such that if $0<\mu+|b|_{q_*}<\Pi_2$ and $0<\Pi_2\leqslant\Pi_1$, then there exists a number $t_\mu^-$ satisfying $1 < t_\mu^- <\tilde{t}_\mu$ such that $t_\mu^- w_\Omega \in N_{\lambda,\mu}^{(1)}$, where
	\[
	\tilde{t}_\mu =
	\begin{cases}
		\left(\frac{p}{2p-m}\right)^{\frac{1}{m-p}}, & \text{if } \int_\Omega b(x) w_\Omega^q dx \geqslant 0, \\
		\left[ \frac{p}{2p-m} \left(1 + \left(\frac{|\int_\Omega b(x) w_\Omega^q dx|}{\int_\Omega a(x) w_\Omega^m dx} \right)  \right) \right]^{\frac{1}{m-p}}, & \text{if } \int_\Omega b(x) w_\Omega^q dx < 0.
	\end{cases}
	\]
\end{lemma}

\begin{proof}[\bf Proof]
For any  \(t>0\), $tw_\Omega \in N_{\lambda,\mu}$  if and only if there holds
	\[
	t^p\|w_\Omega\|_\lambda^p + \mu t^{2p}\int_{\mathbb{R}^3}\phi_{w_\Omega} w_\Omega^pdx - t^m\int_{\mathbb{R}^3} a(x) w_\Omega^mdx - t^q\int_{\mathbb{R}^3} b(x) w_\Omega^qdx = 0.
	\]
Then, we define
\begin{equation*}
	g(t) = t^{-p}\|w_\Omega\|_\lambda^p - t^{m-2p} \int_{\mathbb{R}^3} a(x) w_\Omega^m dx - t^{q-2p} \int_{\mathbb{R}^3} b(x) w_\Omega^q dx, \quad \forall ~ t>0.
\end{equation*}	
From this it follows that $tw_\Omega\in N_{\lambda,\mu}$, if and only if there holds 
\begin{equation}\label{4.1.10}
	g(t)+\mu\int_{\mathbb{R}^3}\phi_{w_\Omega}w_\Omega^p dx=0.
\end{equation}	
 By \eqref{4.9}, we have
\begin{equation}\label{4.10}
	\begin{aligned}
			g(t) &= t^{-p} \int_\Omega |\nabla w_\Omega|^p dx - t^{m-2p} \int_\Omega a(x) w_\Omega^m dx - t^{q-2p} \int_\Omega b(x) w_\Omega^q dx\\
&= t^{q-2p} \left( t^{p-q} - t^{m-q} - C_{b,\Omega} \right) \int_\Omega a(x) w_\Omega^m dx,
	\end{aligned}
\end{equation}	
	where
	\[
	C_{b,\Omega} :=  \frac{\int_\Omega b(x) w_\Omega^q dx}{\int_\Omega a(x) w_\Omega^m dx}   \leqslant    \frac{|b|_{q_*}}{S_{m,\Omega}^{q}(\int_\Omega a(x) w_\Omega^m dx)^{\frac{p-q}{p}}}.
	\]
Let
$h(t)=t^{p-q}-t^{m-q}-C_{b,\Omega}$. Thus, it follows from $h'(t)=0$  that
	\[
	0 < t_1 = \left( \frac{p-q}{m-q} \right)^{\frac{1}{m-p}} < 1,
	\]
	and
	\begin{align*}
		g(t_1) &= t_1^{q-2p} \left( t_1^{p-q} - t_1^{m-q} - C_{b,\Omega} \right)\int_\Omega a(x) w_\Omega^m dx  \\
		&= t_1^{q-2p} \left[ \frac{m-p}{m-q} \left(\frac{p-q}{m-q}\right)^{\frac{p-q}{m-p}}- C_{b,\Omega}  \right]  \int_\Omega a(x) w_\Omega^m dx \\
		&> 0,
	\end{align*}
	provided that
	\[
	|b|_{q_*} < \Gamma_3 := \frac{(m-p)S_{m,\Omega}^{q}}{m-q} \left( \frac{p-q}{m-q} \right)^{\frac{p-q}{m-p}} \left(\int_\Omega a(x) w_\Omega^m dx\right)^{\frac{p-q}{p}}.
	\]
	Next we  consider this problem in the following two cases.
	
	Case (I): $\int_\Omega b(x) w_\Omega^q dx \geqslant 0$. 
Let
	\[
	\widetilde{g}(t) = t^{-p}\|w_\Omega\|_\lambda^p - t^{m-2p}\int_{\mathbb{R}^3}a(x)w_\Omega^m dx, ~\quad \forall ~ t>0.
	\]
Obviously, $g(t)\leqslant \widetilde{g}(t)$ for all $t>0$. Using \eqref{4.9} and \eqref{4.10}, one has,
	\begin{align*}
		\widetilde{g}(t) &= t^{-p}\int_\Omega|\nabla w_\Omega|^pdx - t^{m-2p}\int_\Omega a(x)w_\Omega^m dx \\
		&=  \left( t^{-p} - t^{m-2p} \right)\int_\Omega a(x) w_\Omega^m dx.
	\end{align*}
Clearly,  
	\begin{align*}
	 \widetilde{g}(1)=0,~~~  \lim_{t\to0^+} \widetilde{g}(t)=+\infty,~~~  \lim_{t\to\infty} \widetilde{g}(t)=0.
\end{align*}
Thus, we derive that
	\begin{align}\label{4.4.1}
	\inf_{t>0}\widetilde{g}(t)=\widetilde{g}(t_2)=-\frac{m-p}{2p-m}\left(\frac{p}{2p-m}\right)^{\frac{-p}{m-p}}\int_\Omega a(x) w_\Omega^m dx <0,
\end{align}
where
	$$t_2 = \left(\frac{p}{2p-m}\right)^{\frac{1}{m-p}} > 1.$$
This implies that $\widetilde{g}(t)$ is decreasing on $(0,t_2)$ and increasing on $(t_2,\infty)$. 
By \eqref{4.4.1}, we get
	\begin{align*}
		-\frac{\widetilde{g}(t_2)}{\int_{\mathbb{R}^3}\phi_{w_\Omega} w_\Omega^p dx}&=\frac{m-p}{2p-m}\left(\frac{p}{2p-m}\right)^{\frac{-p}{m-p}}\frac{\int_\Omega a(x) w_\Omega^m dx }{\int_{\mathbb{R}^3} \phi_{w_\Omega}w_\Omega^pdx}\\[3mm]
		&\geqslant \frac{m-p}{2p-m}\left(\frac{p}{2p-m}\right)^{\frac{-p}{m-p}} \frac{\bar{S}_{2,\Omega}^2S_{\frac{6p}{5},\Omega}^{2p}}{\left(\int_\Omega|\nabla w_\Omega|^pdx\right)^2}\int_\Omega a(x) w_\Omega^m dx  \\[3mm]
		&=\frac{m-p}{2p-m}\left(\frac{p}{2p-m}\right)^{\frac{-p}{m-p}} \frac{\bar{S}_{2,\Omega}^2S_{\frac{6p}{5},\Omega}^{2p}}{\int_\Omega a(x) w_\Omega^m dx }\\[3mm]
		&:=\Theta_{1}.
	\end{align*}
Then for every $0<\mu<\Theta_1$, we have
	\[
	\widetilde{g}(t_2) <-\mu\int_{\mathbb{R}^3}\phi_{w_\Omega}w_\Omega^pdx.
	\]
	Since $g(t_2) \leqslant \widetilde{g}(t_2)$, we get
	 $$g(t_2) < -\mu\int_{\mathbb{R}^3}\phi_{w_\Omega}w_\Omega^pdx.$$
Thus, there exists  two  positive constants $t_\mu^-$ and $t_\mu^+$ satisfying
$
t_\mu^- < t_2 < t_\mu^+
$
such that	
\[
	g(t_\mu^{\pm}) + \mu\int_{\mathbb{R}^3}\phi_{w_\Omega}w_\Omega^p dx = 0,
	\]
and 
$$g'(t_\mu^-)<0,~~~g'(t_\mu^+)>0.$$ 
Hence, it follows from \eqref{4.1.10} that $t_\mu^\pm w_\Omega\in N_{\lambda,\mu}$.
Besides, by \eqref{4.3} and \eqref{4.10}, one has,
	\[
	K_{\lambda, tw_\Omega}''(1) = t^{2p+1}g'(t),~~~\forall ~t>0.
	\]
This implies that $K_{\lambda, t_\mu^- w_\Omega}''(1)<0$ and $K_{\lambda, t_\mu^+ w_\Omega}''(1)>0$. Therefore,  we derive that
	\begin{align}\label{4.1.11}
		t_\mu^- w_\Omega\in N_{\lambda,\mu}^-, ~~~t_\mu^+ w_\Omega\in N_{\lambda,\mu}^+.
		\end{align}	
	
		Next we prove that $t_\mu^- w_\Omega \in N_{\lambda,\mu}^{(1)}$.
For $\int_\Omega b(x) w_\Omega^q dx \geqslant 0$, we have
	\begin{align}\label{4.1.19}
			\begin{aligned}
	\mathcal{J}_{\lambda,\mu}(t_\mu^- w_\Omega) 
		&= \frac{1}{2p}(t_\mu^-)^p \int_\Omega|\nabla w_\Omega|^pdx +\frac{2p-m}{2pm}(t_\mu^-)^m \int_\Omega a(x)w_\Omega^mdx\\[2mm]
		&\quad -\frac{2p-q}{2pq}(t_\mu^-)^q \int_\Omega b(x)w_\Omega^q dx\\[2mm]
		&\leqslant  \frac{1}{2p}(t_\mu^-)^p \int_\Omega|\nabla w_\Omega|^pdx +\frac{2p-m}{2pm}(t_\mu^-)^m \int_\Omega a(x)w_\Omega^mdx \\[2mm]
		&= \frac{1}{2p}(t_\mu^-)^p \left[ 1 - \frac{2p-m}{m}(t_\mu^-)^{m-p} \right]  \int_\Omega a(x)w_\Omega^mdx.
		\end{aligned}
	\end{align}
Define
	\[
	\eta(t) = \frac{t^p}{2p} \left( 1 -  \frac{2p-m}{m} t^{m-p} \right), \quad 	1 < t  < t_2.
	\]
	A simple calculation shows that
	\[
	\max_{0 < t \leqslant t_2} \eta(t) = \eta(t_2) = \frac{m-p}{2pm} \left( \frac{p}{2p-m} \right)^{\frac{p}{m-p}}.
	\]
	Thus, it is conclude that
	\[
	\max_{0 < t \leqslant t_2} \eta(t) \int_\Omega a(x) w_\Omega^m dx < \frac{m-p}{pm} A^p(p,m) \int_\Omega a(x) w_\Omega^m dx.
	\]
Combining this with \eqref{4.1.11}--\eqref{4.1.19}, we obtain
	\[
	\begin{aligned}
		\mathcal{J}_{\lambda,\mu}(t_\mu^- w_\Omega) &< \frac{m-p}{pm} A^p(p,m)(1+ B(p,q,m)) \int_\Omega a(x) w_\Omega^m dx \\
		&= \frac{m-p}{pm} M(p,q,m),
	\end{aligned}
	\]
which leads to $t_\mu^- w_\Omega \in N_{\lambda,\mu}^{(1)}$,

	Case (II): $\int_\Omega b(x) w_\Omega^q dx < 0$. We write
	\begin{align*}
		g(t) &= t^{-p} \int_\Omega |\nabla w_\Omega|^p dx - t^{m-2p} \int_\Omega a(x) w_\Omega^m dx + t^{q-2p} \left| \int_\Omega b(x) w_\Omega^q dx \right| \\
		&=  \left( t^{-p} -  t^{m-2p} \right)\int_\Omega a(x) w_\Omega^m dx  + t^{q-2p} \left| \int_\Omega b(x) w_\Omega^q dx \right|.
	\end{align*}
	Define
	\[
	\widehat{g}(t) = t^{-p} \left(\int_\Omega a(x) w_\Omega^m dx  + \left| \int_\Omega b(x) w_\Omega^q dx \right| \right) - t^{m-2p} \int_\Omega a(x) w_\Omega^m dx, ~\quad \forall ~ t>0.
	\]
	It is easy to verify that
	\begin{equation}
		g(t) < \widehat{g}(t),~ \quad \forall~  t>1, \label{4.11}
	\end{equation}
	and 
	$$\widehat{g}(1)=g(1)>0,~~~  \lim_{t\to0^+} \widehat{g}(t)=+\infty,~~~ \lim_{t\to\infty} \widehat{g}(t) =0.$$
Thus, we derive that
	\begin{equation}
		\inf_{t>0} \widehat{g}(t) = \widehat{g}(t_3) = -\frac{m-p}{(2p-m)t_3^{p}} \left(\int_\Omega a(x) w_\Omega^m dx  + \left| \int_\Omega b(x) w_\Omega^q dx \right| \right) < 0, \label{4.13}
	\end{equation}
	where 
		\begin{equation}
		t_3 := \left[ \frac{p}{2p-m} \left( 1 +  \frac{\left| \int_\Omega b(x) w_\Omega^q dx \right|}{\int_\Omega a(x) w_\Omega^m dx } \right) \right]^{\frac{1}{m-p}} > 1. \label{4.12}
	\end{equation}
	Moreover, $\widehat{g}(t)$ is decreasing on $(0,t_3)$ and is increasing on $(t_3,\infty)$. 
By \eqref{4.13}, it is concluded that
	\begin{align*}
		-\frac{\widehat{g}(t_3)}{\int_{\mathbb{R}^3}\phi_{w_\Omega} w_\Omega^p dx} 	&=\frac{m-p}{(2p-m)t_3^{p}} \left(\int_\Omega a(x) w_\Omega^m dx  + \left| \int_\Omega b(x) w_\Omega^q dx \right| \right) \frac{1}{\int_{\mathbb{R}^3} \phi_{w_\Omega}w_\Omega^pdx}\\[3mm]
		&\geqslant \frac{m-p}{(2p-m)t_3^{p}} \left(\int_\Omega a(x) w_\Omega^m dx  + \left| \int_\Omega b(x) w_\Omega^q dx \right| \right)
		\frac{\bar{S}_{2,\Omega}^2S_{\frac{6p}{5},\Omega}^{2p}}{\left(\int_\Omega|\nabla w_\Omega|^pdx\right)^2}\\[3mm]
		&= \frac{m-p}{(2p-m)t_3^{p}} \left( 1 +  \frac{\left| \int_\Omega b(x) w_\Omega^q dx \right|}{\int_\Omega a(x) w_\Omega^m dx } \right)
		\frac{\bar{S}_{2,\Omega}^2S_{\frac{6p}{5},\Omega}^{2p}}{\int_\Omega|\nabla w_\Omega|^pdx}\\[3mm]
		&=\frac{m-p}{2p-m}\left(\frac{p}{2p-m}\right)^{\frac{-p}{m-p}}\left( 1 +  \frac{\left| \int_\Omega b(x) w_\Omega^q dx \right|}{\int_\Omega a(x) w_\Omega^m dx } \right) ^{\frac{m-2p}{m-p}} \frac{\bar{S}_{2,\Omega}^2S_{\frac{6p}{5},\Omega}^{2p}}{\int_\Omega a(x) w_\Omega^m dx }\\[3mm]
		&:=\Theta_{2}.
	\end{align*}
	Hence, for every $0<\mu<\Theta_{2}$, we have
	\[
	\widehat{g}(t_3)<-\mu\int_{\mathbb{R}^3}\phi_{w_\Omega}w_\Omega^pdx.
	\]
Together with \eqref{4.11} and \eqref{4.12}, we get
	\begin{equation}
		g(t_3) < \widehat{g}(t_3) < -\mu \int_{\mathbb{R}^3} \phi_{w_\Omega} w_\Omega^p dx. \label{4.14}
	\end{equation}
 Thus, there exists  two  positive constants $t_\mu^-$ and $t_\mu^+$ satisfying
	\begin{equation}
	1 < t_\mu^- < t_3 < t_\mu^+. \label{4.15}
\end{equation}
such that
	\[
	g( t_\mu^{\pm}) + \mu\int_{\mathbb{R}^3}\phi_{w_\Omega}w_\Omega^p dx = 0,
	\]
and
$$g'(t_\mu^-)<0,~~~g'(t_\mu^+)>0.$$
Hence, it follows from \eqref{4.1.10} that $t_\mu^\pm w_\Omega\in N_{\lambda,\mu}$.
Therefore, similar to the proof of \eqref{4.1.11}, we conclude that $t_\mu^\pm w_\Omega\in N_{\lambda,\mu}^\pm$.

	Next we prove that $t_\mu^- w_\Omega \in N_{\lambda,\mu}^{(1)}$. For $\int_\Omega b(x) w_\Omega^q dx < 0$, we have
	\begin{equation}\label{4.16}
		\begin{aligned}
			&\mathcal{J}_{\lambda,\mu}(t_\mu^- w_\Omega)\\[2mm]
			&= \frac{1}{2p}(t_\mu^-)^p \int_\Omega|\nabla w_\Omega|^pdx -\frac{2p-m}{2pm}(t_\mu^-)^m \int_\Omega a(x)w_\Omega^mdx
			-\frac{2p-q}{2pq}(t_\mu^-)^q \int_\Omega b(x)w_\Omega^q dx\\[2mm]
			&= \frac{1}{2p}(t_\mu^-)^p  \left[ 1 - \frac{2p-m}{m}(t_\mu^-)^{m-p} \right] \int_\Omega a(x)w_\Omega^mdx
			  + \frac{2p-q}{2pq}(t_\mu^-)^q \left| \int_\Omega b(x)w_\Omega^q dx \right| \\[2mm]
			&:= \mathcal{J}_1 + \mathcal{J}_2,
		\end{aligned}
	\end{equation}
	where
	\[
	\mathcal{J}_1 = \frac{1}{2p}(t_\mu^-)^p  \left[ 1 - \frac{2p-m}{m}(t_\mu^-)^{m-p} \right] \int_\Omega a(x)w_\Omega^mdx,
	\]
	and
	\begin{equation}\label{4.4.26}
			\begin{aligned}
	\mathcal{J}_2 &=\frac{2p-q}{2pq}(t_\mu^-)^q \left| \int_\Omega b(x)w_\Omega^q dx \right|\\
		&\leqslant \frac{2p-q}{2pq}\left[ \frac{p}{2p-m} \left( 1 +  \frac{\left| \int_\Omega b(x) w_\Omega^q dx \right|}{\int_\Omega a(x) w_\Omega^m dx } \right) \right]^{\frac{q}{m-p}}  \left| \int_\Omega b(x)w_\Omega^q dx \right|.
	\end{aligned}
		\end{equation}
On the one hand, we  define a function 
	\[
	S(t) = \frac{t^p}{2p} \left( 1 - \frac{2p-m}{m} t^{m-p} \right), \quad 	1 < t < t_3.
	\]
	By a direct calculation, we obtain
	\[
	\max_{0 < t \leqslant  t_3} S(t) = S(t_*) = \frac{m-p}{2pm} \left( \frac{p}{2p-m} \right)^{\frac{p}{m-p}},
	\]
	where
	\begin{equation*}
		t_* = \left( \frac{p}{2p-m} \right)^{\frac{1}{m-p}} \in (1, t_3). 
	\end{equation*}
	Hence,
	\[
	\max_{0 < t \leqslant t_3} S(t) \int_\Omega a(x)w_\Omega^m dx < \frac{(m-p)A^p(p,m)}{pm} \|w_{\Omega}\|_\lambda^p,
	\]
this implies that
	\begin{equation}
		\mathcal{J}_1  < \frac{(m-p)}{pm}A^p(p,m) \|w_{\Omega}\|_\lambda^p. \label{4.18}
	\end{equation}
	On the other hand, by $0 < \mu + |b|_{q_*} < \Pi_2$ and \eqref{4.4.26}, it holds
	\begin{equation}
		\begin{aligned}
			\mathcal{J}_2 &\leqslant \frac{2p-q}{2pq} \left[ \frac{p}{2p-m} \left( 1 + \frac{\left|\int_\Omega b(x) w_\Omega^q dx\right|}{\int_\Omega a(x) w_\Omega^m dx} \right) \right]^{\frac{q}{m-p}} \left| \int_\Omega b(x) w_\Omega^q dx \right| \\[2mm]
			&\leqslant  \frac{2p-q}{2pq}  \left[\frac{p}{2p-m}  \left( 1 + \frac{|b|_{q_*} |w_\Omega|_{m}^q}{\int_\Omega a(x) w_\Omega^m dx} \right) \right]^{\frac{q}{m-p}} |b|_{q_*}S_{m,\Omega}^{-q}\left(\int_\Omega a(x) w_\Omega^m dx\right)^{\frac{q}{p}}  \\[2mm]
			&\leqslant
			\begin{cases}
				\frac{2p-q}{pq} \left( \frac{p}{2p-m} \right)^{\frac{q}{m-p}} |b|_{q_*}  S_{m,\Omega}^{-q} \left( \int_\Omega a(x) w_\Omega^m dx \right)^{\frac{q}{p}} & \text{if } q \leqslant m-p \\[2mm]
				\frac{2p-q}{2pq} \left( \frac{2p}{2p-m} \right)^{\frac{q}{m-p}} |b|_{q_*}S_{m,\Omega}^{-q} \left( \int_\Omega a(x) w_\Omega^m dx \right)^{\frac{q}{p}} & \text{if } q >  m-p
			\end{cases} \\[2mm]
			&\leqslant \frac{2p-q}{pq} \left(\frac{2p}{2p-m} \right)^{\frac{q}{m-p}} |b|_{q_*} S_{m,\Omega}^{-q} \left( \int_\Omega a(x) w_\Omega^p dx \right)^{\frac{q}{p}} \\[2mm]
			&\leqslant \frac{m-p}{pm} A^p(p,m) B(p,q,m) \int_\Omega a(x) w_\Omega^m dx.
		\end{aligned}
		\label{4.19}
	\end{equation}
	It follows from \eqref{4.16}--\eqref{4.19} that
	\[
	\begin{aligned}
		\mathcal{J}_{\lambda,\mu}(t_\mu^- w_\Omega) &< \frac{m-p}{pm} A^p(p,m)(1+ B(p,q,m)) \int_\Omega a(x) w_\Omega^m dx \\
		&= \frac{m-p}{pm} M(p,q,m).
	\end{aligned}
	\]
	which implies that $t_\mu^- w_\Omega \in N_{\lambda,\mu}^{(1)}$.	
	The proof is complete.
\end{proof}

\section{A solution with positive energy}\label{sec5}

	In this section, we  will  prove that if  $\lambda$ is large enough, then system \eqref{107} has a positive
	energy solution for  sufficiently small $\mu$  and $|b|_{q_*}$.
Following \cite{1992Ta}, we have the following result.

\begin{lemma}\label{lem5.1}
Let $\frac{3}{2}<p<3$ and  $1<q<p<m<2p$.  Assume that  $(H_{1})$, $(H_{3})$ and $(H_{4})$ hold. There exist a constant $\sigma>0$ and a differentiable function $t^*: B(0,\sigma)\subset W_\lambda \to\mathbb{R}^+$ such that  for each $u \in N_{\lambda,\mu}^{(1)}$, there hold
	\[
	t^*(0)=1, ~~ t^*(v)(u-v)\in N_{\lambda,\mu}^{(1)}, ~~\forall~v\in B(0,\sigma),
	\]
 and  
	\begin{align*}
		\langle (t^*)'(0),\varphi\rangle =\frac{\psi_u(\varphi)}{\Phi(u)}, ~~\forall ~ \varphi\in W_\lambda,
	\end{align*}
	where
	\begin{align*}
		\psi_u(\varphi)=& p\int_{\mathbb{R}^3}(|\nabla u|^{p-2}\nabla u\nabla\varphi+\lambda V(x)|u|^{p-2}u\varphi)dx+2p \mu\int_{\mathbb{R}^3}\phi_u |u|^{p-2}u\varphi dx\\
		&-m\int_{\mathbb{R}^3}a(x)|u|^{m-2}u\varphi dx-q\int_{\mathbb{R}^3}b(x)|u|^{q-2}u\varphi dx,
	\end{align*}
	and
	$$\Phi(u)=-p\|u\|_\lambda^p+(2p-m)\int_{\mathbb{R}^3}a(x)|u|^mdx+(2p-q)\int_{\mathbb{R}^3}b(x)|u|^qdx.$$
\end{lemma}

\begin{proof}
	For  $u\in N_{\lambda,\mu}^{(1)}$, we define the function $F_u:\mathbb{R}\times W_\lambda\to\mathbb{R}$ by
	\[
	F_u(t,v)=\langle \mathcal{J}_{\lambda,\mu}'(t(u-v)),t(u-v)\rangle.
	\]
By a  direct calculation, one has, 
	\[
	\begin{aligned}
		\displaystyle	F_u(t,v)&=	\displaystyle t^p\|u-v\|_\lambda^p + \mu t^{2p}\int_{\mathbb{R}^3}\phi_{u-v}|u-v|^p\,dx \\
		&\quad 	\displaystyle -t^m\int_{\mathbb{R}^3}a(x)|u-v|^mdx-t^q\int_{\mathbb{R}^3}b(x)|u-v|^qdx.
	\end{aligned}
	\]
	Since $u\in N_{\lambda,\mu}^{(1)}$, it is easy to see that $F_u(1,0)=\langle \mathcal{J}_{\lambda,\mu}'(u),u\rangle=0$. Combining this with \eqref{4.3}, we derive that
	\begin{align*}
		\frac{\partial F_u}{\partial t}(1,0)&=p\|u\|_\lambda^p+2p \mu\int_{\mathbb{R}^3}\phi_u |u|^pdx-m\int_{\mathbb{R}^3}a(x)|u|^mdx-q\int_{\mathbb{R}^3}b(x)|u|^qdx \\
		&=-p\|u\|_\lambda^p-(m-2p)\int_{\mathbb{R}^3}a(x)|u|^mdx-(q-2p)\int_{\mathbb{R}^3}b(x)|u|^qdx \\
		&<0.
	\end{align*}
	Applying the implicit function theorem, we conclude that there exist a constant $\sigma_0>0$ and a differentiable function $t^*:B(0,\sigma_0)\subset W_\lambda \to\mathbb{R}$ with $t^*(0)=1$ such that
	\[
	F_u(t^*(v),v)=0,\quad \forall~v\in B(0,\sigma_0).
	\]
This yields that
	\[
	\langle \mathcal{J}_{\lambda,\mu}'(t^*(v)(u-v)),t^*(v)(u-v)\rangle=0,\quad \forall~v\in B(0,\sigma_0).
	\]
Thus, for any $\varphi\in  W_\lambda$, it is  concluded that
	\[
	\langle (t^*)'(0),\varphi\rangle = \frac{\psi_u(\varphi)}{\Phi(u)}.
	\]
	 By the  continuity of the map $t^*$ and of the functional $\mathcal{J}_{\lambda,\mu}$, we may choose $\sigma\in(0,\sigma_0]$ sufficiently small such that 
	$\|t^*(v)(u-v)\|_\lambda<\widehat D_1$ for all $v\in B(0,\sigma)$, and further,
	\begin{align*}
		K_{\lambda,t^*(v)(u-v)}''(1) &= -p\|t^*(v)(u-v)\|_{\lambda}^p-(m-2p)\int_{\mathbb{R}^3}a(x)|t^*(v)(u-v)|^mdx \\
		&\quad -(q-2p)\int_{\mathbb{R}^3}b(x)|t^*(v)(u-v)|^qdx \\
		&<0,
	\end{align*}
	and
	\[
	\mathcal{J}_{\lambda,\mu}(t^*(v)(u-v)) < \frac{m-p}{pm}M(p,q,m).
	\]
Therefore, we derive that $t^*(v)(u-v)\in N_{\lambda,\mu}^{(1)}$ for all $v\in B(0,\sigma)$. This completes the proof.
\end{proof}

By the Lemmas \ref{lem4.2}--\ref{lem4.4}, we define
	\begin{equation}\label{5.2.1}
c^{**}=\inf_{u\in N_{\lambda,\mu}^{(1)}}\mathcal{J}_{\lambda,\mu}(u).
	\end{equation}
Combining this with the fact of $N_{\lambda,\mu}^{(1)}\subset N_{\lambda,\mu}^-$, we derive that
\begin{equation}
	0<D\leqslant c^{**}<\frac{m-p}{pm}M(p,q,m). \label{5.1}
\end{equation}

\begin{lemma}\label{lem5.2}
 Let $\frac{3}{2}<p<3$ and $1<q<p<m<2p$.	Assume that  $(H_{1})$, $(H_{3})$ and $(H_{4})$ hold. Then there exists a sequence $\{u_n\}\subset N_{\lambda,\mu}^{(1)}$ such that
	\begin{equation}
		\mathcal{J}_{\lambda,\mu}(u_n)= c^{**}+o(1),\quad  \mathcal{J}_{\lambda,\mu}'(u_n)=o(1)\text{ in }W_\lambda^{-1}. \label{5.2}
	\end{equation}
\end{lemma}

\begin{proof}
By \eqref{5.2.1}, using the  Ekeland's variational principle \cite{1974Ek}, there exists a minimizing sequence $\{u_n\}\subset N_{\lambda,\mu}^{(1)}$ such that
	\[
 c^{**}   \leqslant	\mathcal{J}_{\lambda,\mu}(u_n)\leqslant c^{**}+\frac{1}{n},
	\]
	and
	\begin{equation}
		\mathcal{J}_{\lambda,\mu}(u_n)\leqslant \mathcal{J}_{\lambda,\mu}(w)+\frac{1}{n}\|w-u_n\|_\lambda,\quad \forall ~w\in N_{\lambda,\mu}^{(1)}. \label{5.3}
	\end{equation}
	Applying Lemma \ref{lem5.1} with $u=u_n$, there exists a function $t_n^*:B(0,\sigma_n)\to\mathbb{R}$ for some $\sigma_n>0$ such that $t_n^*(w)(u_n-w)\in N_{\lambda,\mu}^{(1)}$. For  $u\in W_\lambda \backslash \{0\}$,	 we take  $\delta\in (0,\sigma_n)$  and set
	\[
	w_\delta=\frac{\delta u}{\|u\|_\lambda},\quad z_\delta=t_n^*(w_\delta)(u_n-w_\delta).
	\]
	This implies that $z_\delta\in N_{\lambda,\mu}^{(1)}$. Substituting $w=z_\delta$ into \eqref{5.3}, it is easy to obtain
	\[
	\mathcal{J}_{\lambda,\mu}(z_\delta)-\mathcal{J}_{\lambda,\mu}(u_n)\geqslant-\frac{1}{n}\|z_\delta-u_n\|_\lambda.
	\]
	Applying the mean value theorem,  we get
	\begin{equation}
		\langle \mathcal{J}_{\lambda,\mu}'(u_n),z_\delta-u_n\rangle+o(\|z_\delta-u_n\|_\lambda)\geqslant-\frac{1}{n}\|z_\delta-u_n\|_\lambda. \label{5.4}
	\end{equation}
Obviously,
\begin{equation}\label{5.4.4}
	z_\delta-u_n = \big(t_n^*(w_\delta)-1\big)(u_n-w_\delta) - w_\delta,
	\end{equation}
and substituting  this into \eqref{5.4}, one has,
\begin{eqnarray}\label{5.5}
	\begin{array}{ll}
	&	\displaystyle  -\langle \mathcal{J}_{\lambda,\mu}'(u_n), w_\delta\rangle
		+ (t_n^*(w_\delta)-1)\langle \mathcal{J}_{\lambda,\mu}'(u_n),u_n-w_\delta\rangle  \\[2mm]
		& \displaystyle \geqslant -\frac{1}{n}\|z_\delta-u_n\|_\lambda
		+ o(\|z_\delta-u_n\|_\lambda). 
 \end{array}
\end{eqnarray}
 Moreover, for \(z_\delta\in N_{\lambda,\mu}^{(1)}\), we have
    \[
    \begin{aligned}
	(t_n^*(w_\delta)-1)\langle \mathcal{J}_{\lambda,\mu}'(u_n), u_n-w_\delta\rangle
	&= (t_n^*(w_\delta)-1)\langle \mathcal{J}_{\lambda,\mu}'(z_\delta), (u_n-w_\delta)\rangle \\
	&\quad ~ + (t_n^*(w_\delta)-1)\langle \mathcal{J}_{\lambda,\mu}'(u_n)-\mathcal{J}_{\lambda,\mu}'(z_\delta), u_n-w_\delta\rangle \\
	&= (t_n^*(w_\delta)-1)\langle \mathcal{J}_{\lambda,\mu}'(u_n)-\mathcal{J}_{\lambda,\mu}'(z_\delta), u_n-w_\delta\rangle .
    \end{aligned}
    \]
    Substituting this identity into \eqref{5.5} and recalling \(w_\delta = \frac{\delta u}{\|u\|_\lambda}\), we obtain
    \begin{align}\label{5.4.5}
    	 \begin{aligned}
	\frac{\langle \mathcal{J}_{\lambda,\mu}'(u_n), u \rangle}{\|u\|_\lambda}
	\leqslant& \frac{\|z_\delta-u_n\|_\lambda}{\delta n}+\frac{o(\|z_\delta-u_n\|_\lambda)}{\delta} \\
	&+\frac{(t_n^*(w_\delta)-1)}{\delta}\langle \mathcal{J}_{\lambda,\mu}'(u_n)-\mathcal{J}_{\lambda,\mu}'(z_\delta),u_n-w_\delta\rangle.
    \end{aligned}
    \end{align}
  Using \eqref{5.4.4} and  $\|w_\delta\|_\lambda\ = \delta$, we get
    \[
    \begin{aligned}
	\|z_\delta-u_n\|_\lambda
	&\leqslant |t_n^*(w_\delta)-1| \, \|u_n-w_\delta\|_\lambda + \delta \\
	&\leqslant C|t_n^*(w_\delta)-1| + \delta ,
    \end{aligned}
    \]
    where $C>0$ is a constant independent of $\delta$. 
By $t_n^*(0) = 1$, it follows that
    \[
    \lim_{\delta\to0}\frac{|t_n^*(w_\delta)-1|}{\delta}
    =\lim_{\delta\to0}\frac{|t_n^*(w_\delta)-t_n^*(0)|}{\delta}
    \leqslant\|(t_n^*)'(0)\|_{W^{-1}_\lambda}\leqslant C.
    \]
     Passing to the limit as $\delta \to 0$ in \eqref{5.4.5}, we conclude that
    \[
    \frac{\langle \mathcal{J}_{\lambda,\mu}'(u_n), u \rangle}{\|u\|_\lambda}\leqslant\frac{C}{n},
    \]
    which yields \eqref{5.2}. This completes the proof.
\end{proof}


\begin{theorem}\label{thm5.4}
 Let $\frac{3}{2}<p<3$ and $1<q<p<m<2p$.	Assume that  $(H_{1})$-$(H_{4})$ hold.  There exists a constant $\Pi_0\leqslant \Pi_2$ such that if $0<\mu+|b|_{q_*}<\Pi_0$ and  $\lambda>0$ is sufficiently large, then system \eqref{107}  admits a positive  solution $u_{\lambda,\mu}^-\in N_{\lambda,\mu}^{(1)}$  with  positive energy and  	$\rho_0 < \|u_{\lambda,\mu}^- \|_\lambda < \widehat{D}_1$, where $\Pi_2>0$ is given in Lemma \ref{lem4.2}.
\end{theorem}

\begin{proof}
	By Lemma \ref{lem5.2}, there exists a bounded sequence $\{u_n\}\subset N_{\lambda,\mu}^{(1)}$ satisfying \eqref{5.2} and
		\[
           \rho_0  <	\|u_n\|_\lambda<\widehat{D}_1.
	\]
Combining this with \eqref{5.1} and Lemma \ref{lem2.3}, we derive that if  $\lambda>0$ is sufficiently large,
 then there exists  $u_{\lambda,\mu}^-\in W_\lambda\setminus\{0\}$ such that 
	$u_n \to u_{\lambda,\mu}^- $ strongly in $ W_\lambda$, and 
	\[
	\mathcal{J}_{\lambda,\mu}(u_{\lambda,\mu}^-) =c^{**} > 0, \qquad
	\mathcal{J}_{\lambda,\mu}'(u_{\lambda,\mu}^-) = 0 .
	\]
	By the weak lower semicontinuity of the norm, one has 
	\begin{equation}
		\|u_{\lambda,\mu}^-\|_\lambda\leqslant\liminf_{n\to\infty}\|u_n\|_\lambda\leqslant\widehat{D}_1. \label{5.6}
	\end{equation}	
Together with the argument of \eqref{4.7}--\eqref{4.1.7}, we derive that there exists a constant $\Pi_0\leqslant \Pi_2$ such that if $0<\mu+|b|_{q_*}<\Pi_0$, then 
	\[
	K_{\lambda,u_{\lambda,\mu}^-}''(1)<0.
	\]
This yields that $u_{\lambda,\mu}^- \in N_{\lambda,\mu}^-$. Moreover,  for $t_\mu^- w_\Omega \in N_{\lambda,\mu}^{(1)}$, one has 
$$c^{**} =	\mathcal{J}_{\lambda,\mu}(u_{\lambda,\mu}^-)\leqslant \mathcal{J}_{\lambda,\mu}(t_\mu^- w_\Omega)< \frac{m-p}{pm} M(p,q,m),$$
which and \eqref{5.6} yield that  $u_{\lambda,\mu}^- \in N_{\lambda,\mu}^{(1)}$.
Furthermore,  we have
$$
	\mathcal{J}_{\lambda,\mu}(|u_{\lambda,\mu}^-|) = \mathcal{J}_{\lambda,\mu}(u_{\lambda,\mu}^-)= c^{**}>0.
$$
Therefore, by Lemma \ref{lem4.1}, we  may suppose that $u_{\lambda,\mu}^-$ is  a positive solution of system \eqref{107}.
\end{proof}

\begin{proof}[\bf Proof of Theorem \ref{thm1}] 
By Theorems \ref{thm3.2} and \ref{thm5.4}, there exists a number $\Pi_0>0$ such that for every $0<\mu+|b|_{q_*}<\Pi_0$, system \eqref{107} has at least two positive solutions $(u_{\lambda,\mu}^\pm,\phi_{u_{\lambda,\mu}^\pm})\in W_\lambda \times D^{1,2}(\mathbb{R}^3)$ satisfying
\[
\|u_{\lambda,\mu}^+\|_\lambda < \left(\frac{(p-q)\mathbf{S}_m^m}{(m-q)\|a\|_\infty}\right)^{\frac{1}{m-p}}< \|u_{\lambda,\mu}^-\|_\lambda, 
\]
and
\[
\mathcal{J}_{\lambda,\mu}(u_{\lambda,\mu}^+) < 0 < \mathcal{J}_{\lambda,\mu}(u_{\lambda,\mu}^-),
\]
when  $\lambda > 0$ is sufficiently large.
\end{proof}



\begin{thebibliography}{30}
	
	\footnotesize






\bibitem{1995BW-CPDE} T. Bartsch, Z. Wang, Existence and multiplicity results for some superlinear elliptic problems on $\mathbb{R}^{N}$.  \emph{Comm. Partial Differential Equations,} \textbf{20} (1995), 1725--1741.

\bibitem{1998BF} V. Benci, D. Fortunato, An eigenvalue problem for the Schr\"{o}dinger-Maxwell equations.  \emph{Topol. Methods Nonlinear Anal.,} \textbf{11} (1998), 283--293.

\bibitem{2002BF} V. Benci, D. Fortunato, Solitary waves of the nonlinear Klein-Gordon equation coupled with the Maxwell equations. \emph{  Rev. Math. Phys.,}  \textbf{14} (2002), 409--420.

\bibitem{1992BM} L. Boccardo, F. Murat, Almost everywhere convergence of the gradients of solutions to elliptic and parabolic equations. \emph{Nonlinear Anal.,} \textbf{19}(1992) 581--597.


\bibitem{2003Br}   K. J. Brown, Y. Zhang,  The Nehari manifold for a semilinear elliptic equation with a sign-changing weight function. \emph{J. Differential Equations},  \textbf{193}(2003) 481--499.



\bibitem{2021-CL-AML} S. Chen, L. Li, Infinitely many large energy solutions for the Schr\"{o}dinger-Poisson system with concave and convex nonlinearities.  \emph{Appl. Math. Lett.,} \textbf{112} (2021), 106789.

\bibitem{2023CT-AML}    X. Chen,  C. Tang, Sign-Changing Solutions for Schr\"{o}dinger-Poisson  System With p-Laplacian in $\mathbb{R}^{3}$.   \emph{Appl. Math. Lett.,} \textbf{139} (2023), 108535.


\bibitem{2004DM-PRSEA} T. D'Aprile, D. Mugnai, Solitary waves for nonlinear Klein-Gordon-Maxwell and Schr\"{o}dinger-Maxwell equations. \emph{Proc. Roy. Soc. Edinburgh Sect. A,} \textbf{134}(2004) 893--906.

\bibitem{2004DM-ANS} T. D'Aprile, D. Mugnai, Non-existence results for the coupled Klein-Gordon-Maxwell equations. \emph{Adv. Nonlinear Stud.,}  \textbf{4}(2004) 307--322.

\bibitem{2002d-ANS} P. d'Avenia, Non-radially symmetric solutions of nonlinear Schr\"{o}dinger equation coupled with Maxwell equations. \emph{Adv. Nonlinear Stud.,}  \textbf{2} (2002), 177--192.

\bibitem{2011DPV} P. d'Avenia, A. Pomponio, G. Vaira, Infinitely many positive solutions for a Schr\"{o}dinger-Poisson system.  \emph{Nonlinear Anal.,} \textbf{74}(2011) 5705--5721.



\bibitem{2022DSW-JMAA}   Y. Du, J. Su, C. Wang, On a quasilinear Schr\"{o}dinger-Poisson system. \emph{J. Math. Anal. Appl.,}  \textbf{505}(2022), 125446, 1--14.

\bibitem{2022DSW-CPAA} Y. Du, J. Su, C. Wang, On the critical Schr\"odinger-Poisson system  with $p$-Laplacian. \emph{Commun. Pure Appl. Anal.,} \textbf{21}(2022) 1329--1342.


\bibitem{1974Ek}   I. Ekeland,  On the variational principle. \emph{J. Math. Anal. Appl.},  \textbf{47}(1974) 324--353.








\bibitem{2024HS-JMP} L. Huang, J. Su, Multiple positive solutions of the quasilinear Schr\"{o}dinger-Poisson system with critical exponent in $D^{1,p}(\mathbb{R}^{3})$. \emph{J. Math. Phys., }  \textbf{65} (2024)  051512.


\bibitem{2011JZ-JDE}  Y.S. Jiang, H.S. Zhou, Schr\"{o}dinger-Poissonsystem with steep potential well.  \emph{J. Differential Equations,}  \textbf{251} (2011) 582--608.


\bibitem{2023KLT-JGA} J.-C. Kang,  X.-Q. Liu, C.-L. Tang, Ground state sign-changing solutions for critical Schr\"{o}dinger-Poisson system with steep potential well.  \emph{J. Geom. Anal.,} \textbf{33}(2023), no.2, Paper No. 59, 24 pp.


\bibitem{2017LS-CMA} C. Lei, H. Suo, Positive solutions for a  Schr\"{o}dinger-Poisson system involving concave-convex nonlinearities.  \emph{Comput. Math. Appl.,}  \textbf{74} (2017) 1516--1524.

\bibitem{LL} E.H. Lieb, M. Loss, \emph{Analysis.}  American Mathematical Society, 2001.

\bibitem{2025LH-BMS}  K. Liu, X. He,  Solutions with prescribed mass for the Sobolev critical Schr\"{o}dinger-Poisson system with p-Laplacian.   \emph{Bull. Math. Sci.,}  \textbf{15} (2025)  2550006.

\bibitem{2025LHR-JDE} K. Liu, X. He,  V. D. R\u{a}dulescu, Solutions with prescribed mass for the p-Laplacian Schr\"{o}dinger-Poisson  system with critical growth. \emph{J. Differential Equations,} \textbf{444} (2025) 113570.



\bibitem{2023PJ-ZAMP}  X. Peng, G. Jia, Existence and concentration behavior of solutions for the logarithmic Schr\"{o}dinger-Poisson system with steep potential.  \emph{Zeitschrift f\"{u}r Angewandte Mathematik und Physik,} \textbf{74} (2023) 29.





\bibitem{2006Ruiz} D. Ruiz, The Schr\"{o}dinger-Poisson equation under the effect of a nonlinear local term. \emph{J. Funct. Anal.,}  \textbf{237} (2006)  655--674.


\bibitem{2018SM-AML} M. Shao, A. Mao, Multiplicity of solutions to Schr\"{o}dinger-Poisson system with concave-convex nonlinearities. \emph{Appl. Math. Lett.,} \textbf{83} (2018) 212--218.

\bibitem{2019SM-JMP} M. Shao, A. Mao, Schr\"{o}dinger-Poisson system with concave-convex nonlinearities, \emph{J. Math. Phys.,} \textbf{60} (2019) 061504.

\bibitem{2020SW-NA} J. Sun, T.-F. Wu,  Steep potential well may help Kirchhoff type equations to generate multiple solutions.  \emph{Nonlinear Anal.,} \textbf{190}(2020)  111609.

\bibitem{2021SW-CCM} J. Sun, T.-F. Wu, On Schr\"{o}dinger-Poisson systems involving concave-convex nonlinearities via a novel constraint approach. \emph{Commun. Contemp. Math.,} \textbf{23} (2021), Paper No. 2050048, 25 pp.

\bibitem{2015SSZ-DCDS} M. Sun, J. Su, L. Zhao, Infinitely many solutions for a Schr\"{o}dinger-Poisson system with concave and convex nonlinearities. \emph{Discrete Contin. Dyn. Syst.,} \textbf{35} (2015)  427--440.






\bibitem{1992Ta}   G. Tarantello,  On nonhomogeneous elliptic equations involving critical Sobolev exponent. \emph{Ann. Inst. H. Poincar\'e Anal. Non Lin\'eaire},  \textbf{9}(1992) 281--304.





\bibitem{2013Willem} M. Willem, \emph{Functional analysis.} Fundamentals and applications. Birkh\"{a}user/Springer, New York, 2013.


\bibitem{2025XNL-ZAMP} D. Xiao, T. V. Nguyen, S. Liang, Normalized solutions for critical Schr\"{o}dinger-Poisson system involving p-Laplacian in $\mathbb{R}^3$. \emph{Z. Angew. Math. Phys.,} \textbf{76} (2025)  1--25.


\bibitem{2009YSD-NA}  M. Yang, Z. Shen, Y. Ding, Multiple semiclassical solutions for the nonlinear Maxwell-Schr\"{o}dinger system.  \emph{Nonlinear Anal.,} \textbf{71} (2009) 730--739.

\bibitem{2021YO-JMAA} Z. Yang, Z. Ou, Nodal solutions for Schr\"{o}dinger-Poisson  systems with concave-convex nonlinearities. \emph{J. Math. Anal. Appl.,}  \textbf{499} (2021) 125006.

\bibitem{2023YT-CAM} C. Yang, C. Tang, Sign-changing solutions for the Schr\"{o}dinger-Poisson  system with concave-convex nonlinearities in $\mathbb{R}^3$. \emph{Commun. Anal. Mech.,} \textbf{15}(2023)  638--657.

\bibitem{2020YTW-AMC}  L. Yin, X. Wu,  C. Tang,  Existence and concentration of ground state solutions for critical  Schr\"{o}dinger-Poisson system with steep potential well. \emph{Appl. Math. Comput.} \textbf{374}  (2020) 125035.




\bibitem{2013ZLZ-JDE}  L. Zhao, H. Liu, F. Zhao, Existence and concentration of solutions for the Schr\"{o}dinger-Poisson equations with steep well potential.  \emph{J. Differential Equations, } \textbf{255}(2013) 1--23.
	
\end{thebibliography}
\end{document}